\documentclass[english]{amsart}
\usepackage{amsmath,amssymb,amscd,amsfonts,mathrsfs}
\usepackage[alphabetic]{amsrefs}
\usepackage{tabularx}
\usepackage{latexsym}
\usepackage{color}
\usepackage[all,cmtip,line]{xy}
\usepackage{enumitem}
\usepackage[only,llbracket,rrbracket]{stmaryrd}
\usepackage{todonotes}

\DeclareSymbolFont{euletters}{U}{eur}{m}{n}
\DeclareSymbolFont{eufrakletters}{U}{euf}{m}{n}
\DeclareFontFamily{U}{wncy}{}
    \DeclareFontShape{U}{wncy}{m}{n}{<->wncyr10}{}
    \DeclareSymbolFont{mcy}{U}{wncy}{m}{n}
    \DeclareMathSymbol{\Sha}{\mathord}{mcy}{"58}

\newcommand{\ilim}{\underleftarrow{\text {\rm lim}}}

\newcommand{\Ext}{\text{\rm Ext}}

\newcommand{\Res}{\text{\rm Res}}

\renewcommand{\dim}{\text {\rm dim}}

\newcommand{\Spec}{{\operatorname{Spec\,}}}
\newcommand{\Spf}{\text {\rm Spf \!}}

\newcommand{\Out}{\text{\rm Out}}
\newcommand{\Spin}{\text{\rm Spin}}

\renewcommand{\ker}{\text{\rm ker}\,}

\newcommand{\ord}{{\text{\rm ord}}}

\newcommand{\ad}{{\text{\rm ad}}}
\newcommand{\ed}{{\text{\rm ed}}}

\renewcommand{\top}{{\text{\rm top}}}

\newcommand{\SL}{\text{SL}}

\newcommand{\der}{\text{\rm der}}
\renewcommand{\sc}{\text{\rm sc}}

\newcommand{\Gal}{\text{\rm Gal}}

\newcommand{\GL}{\text{\rm GL}}
\newcommand{\PGL}{\text{\rm PGL}}
\newcommand{\SO}{\text{\rm SO}}
\newcommand{\SU}{\text{\rm SU}}
\newcommand{\GSp}{\text{\rm GSp}}
\newcommand{\PSp}{\text{\rm PSp}}
\newcommand{\Sp}{\text {\rm {Sp}}}
\newcommand{\PO}{\text {\rm {PO}}}

\newcommand{\Sh}{\text{\rm Sh}}
\newcommand{\SSh}{{\mathscr S}}

\newcommand{\A}{\mathcal A}

\renewcommand{\AA}{\mathbb A}

\newcommand{\CC}{\mathbb C}

\DeclareMathSymbol{\fk}\mathord{eufrakletters}{"6B}

\newcommand{\F}{\mathbb F}
\newcommand{\FF}{\mathcal F}

\newcommand{\GG}{\mathbb G}

\renewcommand{\H}{\mathcal H}

\DeclareMathSymbol{\ei}\mathord{euletters}{"69}
\DeclareMathSymbol{\etau}\mathord{euletters}{"1C}
\DeclareMathSymbol{\eiota}\mathord{euletters}{"13}

\DeclareMathSymbol{\eK}\mathord{euletters}{"4B}

\renewcommand{\L}{\mathcal L}

\newcommand{\M}{\mathcal M}

\newcommand{\gm}{\mathfrak m}

\renewcommand{\O}{\mathcal O}

\newcommand{\Q}{\mathbb Q}

\newcommand{\RR}{\mathbb R}

\newcommand{\Z}{\mathbb Z}
\newcommand{\ZZ}{\mathcal Z}

\newcommand{\syn}{\text{\rm syn}}

\newcommand{ \iso} {\overset \sim \longrightarrow}

\newcommand{\Aut}{\text {\rm Aut}}

\newcommand{\lps}{[\![}
\newcommand{\rps}{]\!]}
\newcommand{\RD}{RD}
\newcommand{\SB}{SB}
\newcommand{\onto}{\twoheadrightarrow}
\newcommand{\into}{\hookrightarrow}

\newtheorem{ithm}{Theorem}
\newtheorem{icor}[ithm]{Corollary}
\numberwithin{equation}{subsection}

\newtheorem{question}[ithm]{Question}

\newtheorem{thm}[equation]{Theorem}

\newtheorem{cor}[equation]{Corollary}
\newtheorem{lemma}[equation]{Lemma}

\newtheorem{prop}[equation]{Proposition}

 \theoremstyle{definition}
 \theoremstyle{definition}
 \theoremstyle{remark}

\newtheorem{para}[equation]{\bf}

\theoremstyle{definition}

\begin{document}

\title{The Essential dimension of congruence covers}


\author{Benson Farb, Mark Kisin and Jesse Wolfson}
\address{ Department of Mathematics, University of Chicago}
\email{farb@math.uchicago.edu}
\address{ Department of Mathematics, Harvard }
\email{kisin@math.harvard.edu}
\address{ Department of Mathematics, University of California-Irvine}
\email{wolfson@uci.edu}


\thanks{The authors are partially supported by NSF grants DMS-1811772 (BF), DMS-1601054 (MK) and
DMS-1811846 (JW)}



\begin{abstract}  Consider the algebraic function $\Phi_{g,n}$ that assigns to a general $g$-dimensional abelian variety an $n$-torsion point.  A question first posed by Kronecker and Klein asks: What is the minimal $d$ such that, after a rational change of variables, the function $\Phi_{g,n}$ can be written as an algebraic function of $d$ variables?

Using techniques from the deformation theory of $p$-divisible groups and finite flat group schemes, we answer this question by computing the essential dimension and $p$-dimension of congruence covers of the moduli space of principally polarized abelian varieties.    We apply this result to compute the essential $p$-dimension of congruence covers of the moduli space of genus $g$ curves, as well as its hyperelliptic locus, and of certain locally symmetric varieties.  These results include cases where the locally symmetric variety $M$  is {\em proper}.  As far as we know, these are the first examples of nontrivial lower bounds on the essential dimension of an unramified, non-abelian covering of a proper algebraic variety.
\end{abstract}



\maketitle
\tableofcontents

\section{Introduction}\label{sec:intro}
This article grew out of an attempt to answer questions first raised by Kronecker and Klein.\footnote{See \cite[p. 171]{KleinLetter}, \cite[p. 309]{Kronecker} and \cite{Burkhardt1,Burkhardt2,Burkhardt3}, esp. the footnote to \cite[Ch. 11]{Burkhardt2} on p. 216, and the treatment in \S51-55 that follows.}  Let $K$ be an algebraically closed field of characteristic $0,$ and consider the algebraic function $\Phi_{g,n}$ that assigns to a general $g$-dimensional (principally polarized) abelian $K$-variety an $n$-torsion point.

\begin{question}
\label{question:first}
Let $g,n\geq 2$.  What is the minimum $d$ such that, after a rational change of variables, the function $\Phi_{g,n}$ can be written as an algebraic function of $d$ variables?
\end{question}

We can rephrase Question~\ref{question:first} in more modern language, using the moduli space of principally polarized abelian varieties  and the notion of {\em essential dimension} introduced by Buhler--Reichstein \cite{BR}.
Let $\A_g$ denote the coarse moduli space of $g$-dimensional, principally polarized abelian varieties over $K$, and let $\A_{g,n}^1$ be
the coarse moduli space of pairs $(A,z)$ where $A$ is a principally polarized abelian variety of dimension $g,$ and $z$ is an $n$-torsion point on $A.$

For any finite, generically \'etale map of $K$-schemes, $p:X'\to X$,  define the {\em essential dimension} 
$\ed_K(X'\to X)$, or $\ed_K(X'/X)$ if the map is implicit, to be the minimal $d$ for which, generically, $X'\to X$ is a pullback:
\begin{equation*}
    \xymatrix{
        X' \ar@{-->}[r] \ar[d] & Y' \ar[d] \\
        X \ar@{-->}[r]^f & Y
    }
\end{equation*}
of a finite map $Y'\to Y$ of $d$-dimensional $K$-varieties via a rational map
$f:X\dashrightarrow Y$.  In this case we call $f$ a {\em (rational) compression} of $p.$ 
Question~\ref{question:first} can be rephrased as asking for the value of $\ed_K(\A_{g,n}^1\to\A_g).$ 

Following Klein \cite{KleinIcos}, we can ask a related question, where we allow certain {\em accessory irrationalities}; that is, we can ask for
\[
    \min_{E\to \A_g} \ed_K(\A_{g,n}^1|_E\to E)
\]
for some class of finite, generically \'etale maps $E\to\A_g$.\footnote{Given maps $X'\to X$ and $E\to X$, we use the notations $X'|_E$ and $X'\times_X E$ interchangeably to denote the fiber product.} For example, the {\it essential $p$-dimension} $\ed_K(X'/X; p)$ is defined
(see \cite{ReiYo}, Definition 6.3) as the minimum of $\ed_K(X'\times_X E\to E)$ where $E\to X$  runs over finite, generically \'etale maps of $K$-varieties of degree prime to $p$.
For any map $E \rightarrow X$ one always has $$\ed_K(X'\times_X E\to E) = \ed_K(\tilde X'\times_X E\to E)$$
 where $\tilde X'$ denotes the composite of Galois closures of the connected components of $X'$ 
(see Lemmas \ref{lem:indepGalois}, \ref{lem:indepGaloisII}, cf.~\cite[Lemma 2.3]{BR}). In particular, for any class of maps $E \to \A_g$ the answer to the question above does not change if we replace $\A_{g,n}^1$ 
by $\A_{g,n},$ the coarse moduli space of pairs $(A,\mathcal{B})$ where $A$ is a principally polarized abelian variety of dimension $g,$ and $\mathcal{B}$ is a symplectic basis for the $n$-torsion $A[n].$ (Here and below we fix once and for all an isomorphism $\mu_n(K) \iso \Z/n\Z,$ so that we may speak of a symplectic basis for $A[n].$)

In this paper we apply techniques from the deformation theory of $p$-divisible groups and finite flat group schemes to compute the essential $p$-dimension of congruence covers of certain locally
symmetric varieties, such as $\A_g$.

\begin{ithm}\label{theorem:main1}
Let $g,n\geq 2$, and let $p$ be any prime with $p|n$. Then
    \[
        \ed_K(\A_{g,n}/\A_g; p)= \ed_K(\A_{g,n}/\A_g) = \dim \A_g=\binom{g+1}{2}.
    \]
\end{ithm}

Theorem~\ref{theorem:main1} thus answers Question~\ref{question:first}: the minimal $d$ equals $\binom{g+1}{2}$.  We in fact prove a more general result, that for subvarieties of $\ZZ \subset \A_g$ satisfying some mild technical hypotheses, $\ed_K(\A_{g,n}|_{\ZZ} / \ZZ) = \dim \ZZ.$ More precisely, we prove the following (see Theorem  \ref{thm:essdimAg} below).

\begin{ithm}\label{theorem:main2}
Let $p$ be prime, and let $N\ge 3$ be an integer prime to $p$.  Suppose that $L = \bar \Q_p,$ an algebraic closure of $\Q_p,$ let $\O_L$ be its ring of integers and let $k$ be its residue field.
Let $\ZZ \subset \A_{g,N/\O_L}$ be a locally closed subscheme that is equidimensional and smooth over $\O_L$, and whose special fiber $\ZZ_k$ meets the ordinary locus  $\A_{g,N}^{\ord} \subset\A_{g,N/k}.$ Then
    $$ \ed_L(\A_{g,p}|_{\ZZ_L}/\ZZ_L; p) = \ed_L(\A_{g,p}|_{\ZZ_L}/\ZZ_L )= \dim \ZZ_L. $$
\end{ithm}

We give three applications of Theorem~\ref{theorem:main2}. The first is an analogue of Theorem \ref{theorem:main1} for $\M_g,$ the  coarse moduli space of smooth, proper, genus $g\geq 2$ curves over $K.$  For any integer $n,$ consider the {\em level $n$ congruence cover} $\M_g[n] \to\M_g$, where $\M_g[n]$ denotes the moduli space of pairs $(C,{\mathcal B})$ consisting of a smooth, proper curve $C$ of genus $g,$ together with a symplectic basis $\mathcal B$ for $J(C)[n],$ where $J(C)$ is the Jacobian of $C.$  Applying Theorem~\ref{theorem:main2} and the Torelli theorem, we will deduce the following.

\begin{icor}\label{icor:Mg}
Let $g,n\geq 2$.  Let $p$ be any prime with $p\mid n$. Then
    \[
        \ed_K(\M_g[n]/\M_g)=\ed_K(\M_g[n]/\M_g; p)=\dim \M_g=3g-3.
    \]
\end{icor}

As a second application of Theorem~\ref{theorem:main2}, we answer a question encountered by Burkhardt  in his study of hyperelliptic functions (see in particular \cite[Ch. 11]{Burkhardt2}).  Let $\H_g$ denote the coarse moduli of smooth, hyperelliptic curves of genus $g\geq 2$ over $K.$ For any integer $n,$ consider the {\em level $n$ congruence cover} $\H_g[n] \to\H_g$, where $\H_g[n]$ denotes the moduli space of pairs $(C,{\mathcal B})$ consisting of a hyperelliptic curve $C$ together with a symplectic basis $\mathcal B$ for $J(C)[n].$  Analogously to the case of $\M_g$, we prove the following.

\begin{icor}\label{icor:Hg}
Let $g,n\geq 2$.  Let $p\mid n$ be any odd prime. Then
    \[
        \ed_K(\H_g[n]/\H_g)=\ed_K(\H_g[n]/\H_g; p)=\dim \H_g=2g-1.
    \]
\end{icor}
\noindent
The hypothesis that $p$ is odd in Corollary~\ref{icor:Hg} is necessary; see \ref{para:p=2} below.

Our third application of Theorem~\ref{theorem:main2} generalizes Theorem \ref{theorem:main1} to many locally symmetric varieties.  Recall that a {\em locally symmetric variety} is a variety whose complex points have the form $\Gamma\backslash X$ where $X$ is a Hermitian symmetric domain and $\Gamma$ is an arithmetic  lattice in the corresponding real semisimple Lie group (see \ref{lem:edineq} below).
Attached to $\Gamma$ there is a semisimple algebraic group $G$ over $\Q,$ with $X = G(\RR)/K_{\infty},$ for $K_{\infty} \subset G(\RR)$ a maximal compact subgroup.
By a {\em principal $p$-level covering} $\Gamma_1\backslash X \to \Gamma\backslash X$  we mean that the definition of $\Gamma$ does not involve any congruences at $p,$ and $\Gamma_1 \subset \Gamma$ is the subgroup of elements that are trivial mod $p.$  A sample of what we prove is the following (see Theorem \ref{thm:essdimcong} below for the most general statement).

\begin{ithm}\label{theorem: main3}
With notation as just given, suppose that each irreducible factor of $G^{\ad}_{\RR}$ is the adjoint group of one of $U(n,n),$ $\SO(n,2)$ with $n+2 \neq 8,$ or $\Sp(2n)$ for some positive integer $n.$  If $G$ is unramified at $p,$ (a condition which holds for almost all $p$)
then for any principal $p$-level covering $\Gamma_1\backslash X \rightarrow \Gamma \backslash X$, we have
$$ \ed(\Gamma_1\backslash X \rightarrow \Gamma \backslash X; p) = \dim X.$$
\end{ithm}

In fact our results apply to any Hermitian symmetric domain of classical type, but in general they require a more involved condition on the $\Q$-group $G$ giving rise to $\Gamma;$ they always apply if $G$ splits over $\Q_p.$ Note that these results include cases where the locally symmetric variety $\Gamma\backslash X$ is {\em proper}. As far as we know, these are the first examples of nontrivial lower bounds on the essential dimension of an unramified, non-abelian covering of a proper algebraic variety.  The only prior result for unramified covers of proper varieties of which we are aware is due to Gabber \cite[Appendix]{CT}, who proved that if $\{E'_i\to E_i\}$ is a collection of connected, unramified $\Z/p\Z$ covers of elliptic curves $E_i$, then under certain conditions, the cover $E_1'\times\cdots\times E_r'\to E_1\times\cdots\times E_r$ has essential dimension at $p$ equal to $r$.

When $\Gamma\backslash X$ is proper, the use of fixed-point techniques is precluded by the fact that one cannot use ``ramification at infinity''. This is in analogy with Margulis superrigidity for irreducible lattices $\Gamma$ in higher rank semisimple Lie groups, where, for nonuniform 
$\Gamma$ (equivalently noncompact $\Gamma\backslash X$), unipotents in $\Gamma$ play a crucial role.  When $\Gamma$ is uniform, it contains no unipotents, and new ideas were needed (and were provided only later, also by Margulis).

There are many examples of finite simple groups of Lie type for which our methods give a lower bound on the essential $p$-dimension of a covering of locally symmetric varieties with that group.

\begin{icor}\label{icor:Lieexamples} Let $H$ be a classical, absolutely simple group over $\F_q,$ with $q = p^r.$ Then there is a congruence $H(\F_q)$-cover of locally symmetric varieties $Y' \rightarrow Y$ such that $e:= \ed_K(Y'/ Y; p)$ satisfies :
\begin{itemize}
\item If $H$ is a form of $\PGL_n$ which is split if $n$ is odd, then $e = r\lfloor \frac{n^2}4 \rfloor.$
\item If $H$ is $\PSp_{2n}$ then $e = r(\frac{n^2+n}2).$
\item If $H$ is a split form of  $\PO_{2n}$ then $e = r(\frac{n^2-n}2).$
\item If $H$ is a form of $\PO_n$ and $H$ is not of type $D_4,$ then $e = r(n-2).$
\end{itemize}
\end{icor}

\medskip
\noindent
{\bf Historical Remarks. }
\begin{enumerate}
        \item The study of essential dimension originates in Hermite's study of the quintic \cite{Hermite}, Kronecker's response to this \cite{Kronecker}, and Klein's ``Resolvent Problem'' (see \cite[Lecture IX]{KleinNU} and \cite{KleinIcos} esp. Part II, Ch. 1.7, as well as \cite{Klein87}, \cite{KleinLast}, and more generally the papers \cite[p. 255-506]{KleinCW}; see also \cite{Tschebotarow}). Following Buhler-Reichstein \cite{BR}, the last two decades have seen an array of computations, new methods, generalizations, and applications of this invariant (see \cite{ReICM} or \cite{MerSurv} for recent surveys), but computations to date have largely focused on a different set of problems than those we consider here.
        \item In contrast to Kronecker, Klein also advocated for the consideration of accessory irrationalities as a ``characteristic feature'' of higher degree equations, and called upon his readers to ``fathom the nature and significance of the necessary accessory irrationalities.'' \cite[p. 174, Part II, Ch.1.7]{KleinIcos} While essential dimension at $p$ partially answers this call, a full answer requires the notion of {\em resolvent degree} $\RD(X\to Y)$,  which asks for the minimum $d$ such that an algebraic function admits a {\em formula}\footnote{Recall that a {\em formula} for an algebraic function $\Phi$ on a complex variety $X$ is a tower of finite field extensions $\mathbb{C}(X)=K_0\subset K_1\subset\cdots\subset K_r$, a choice of primitive element $\Psi_i\in K_i$ for $K_i/K_{i-1}$ for each $i$, and an embedding $\mathbb{C}(X)(\Phi)\subset K_r$ of fields over $\mathbb{C}(X)$. 
One can then write $\Phi = R(\Psi_1,\dots, \Psi_r)$ for $R \in K(X)(t_1,\dots, t_r)$ a rational function, which is to say a ``formula'' for $\Phi$ in terms of the $\Psi_i.$}
using only algebraic functions of $d$ or fewer variables (see \cite{FW}). For example,  Klein's icosahedral formula for the quintic \cite{KleinIcos} shows that $\RD(\Phi)=1$ for any algebraic function $\Phi$ with monodromy $A_5$, while the essential p-dimension can equal 2. At present, we do not know of a single example which provably has $\RD(X\to Y)>1$.
\end{enumerate}

\medskip
\noindent
{\bf Idea of Proof. } To prove Theorem \ref{theorem:main1} and its generalization to subvarieties, we use arithmetic techniques, and specifically Serre-Tate theory, which describes the deformation theory of an ordinary abelian variety in characteristic $p$ in terms of its $p$-divisible group. Let $N \geq 3$ be an integer coprime to $p,$ and let $\A$ denote the universal abelian scheme over $\A_{g,N}$ (now considered over $\mathbb Z[\zeta_N][1/N]$). Let $\A[p]$ be its $p$-torsion group scheme, and let $\A_x$ denote the fiber of $\A$ at $x.$ Given a rational compression of $\A_{g,pN}\to\A_{g,N}$ (in characteristic $0$) onto a smaller-dimensional variety, we show that there is an ordinary mod $p$ point $x$ of $\A_{g,N},$ and a tangent direction $t_x$ at $x,$ such that the deformation of $\A_x[p]$ corresponding to $t_x$ is trivial. From this we deduce that the deformation of $\A_x$ corresponding to $t_x$ is trivial, a contradiction.

One might view our method as an arithmetic analogue of the ``fixed point method'' in the theory of essential dimension (see \cite{ReICM}), where the role of fixed points for a group action is now played by wild ramification at a prime. In the fixed point method one usually works over a field where the order of the group is invertible. In contrast, for us the presence of wild ramification plays an essential role.

\bigskip
\noindent
{\bf Acknowledgements.}  We would like to thank Aaron Landesman and Zinovy Reichstein for detailed comments on an earlier version of this paper and for extensive helpful correspondence. We also thank Maxime Bergeron, Mike Fried, Eric Rains, Alexander Polishchuk and Burt Totaro, for various helpful comments. Finally we thank the referee for many useful comments.

\section{Preliminary results}\label{sec:prelim}
\numberwithin{equation}{subsection}
\subsection{Finite \'etale maps}
We begin with some general lemmas on finite \'etale maps.

\begin{para} Let $G \rightarrow G'$ be a map of groups, and $S$ a finite set with an action of $G.$ 
We say the action of $G$ on $S$ {\em lifts} to $G'$ if there is an action of $G'$ on $S$, which induces the 
given action of $G$ on $S.$ We call such a $G'$-action a {\em lifting} of the $G$-action on $S.$ 
We say that the action of $G$ on $S$ {\em virtually lifts} to $G'$ if there is a finite 
index subgroup $G'' \subset G',$ containing the image of $G,$ such that the action of $G$ on $S$ lifts to $G''.$

For a positive integer $n$ we denote by $S^n$ the $n$-fold product equipped with the diagonal action of $G,$ 
and by $\pi_i: S^n \rightarrow S$, $i=1,\dots, n$ the projections.
\end{para}

\begin{lemma}\label{lem:basicsetlemma} Let $G \rightarrow G'$ be a map of groups, and $S$ a finite set with an action of $G.$ 
\begin{enumerate}
\item Suppose the action of $G$ on $S$ admits a lifting to an action of $G'$ on $S,$ and fix 
such a lifting.  Then there exists a finite index subgroup $G'' \subset G'$ containing the image of $G,$ 
such that $G$ and $G''$ have the same image in $\Aut(S).$
\item Let $T \subset S^n$ be a $G$-stable subset such that $\cup_{i=1}^n \pi_i(T) = S.$ Then 
the action of $G$ on $S$ virtually lifts to $G'$ if and only if the action of $G$ on $T$ virtually lifts to $G'.$
\end{enumerate}
\end{lemma}
\begin{proof} Let $N \subset G'$ be the kernel of $G' \rightarrow \Aut(S),$ and $G'' \subset G'$ the subgroup 
generated by $N$ and the image of $G.$ Since $S$ is finite, $N$ and hence $G''$ has finite index in $G',$ 
and satisfies the conditions in (1).

For (2), suppose that the action of $G$ on $S$ virtually lifts to $G'.$ By (1), after replacing $G'$ by a finite index subgroup 
containing the image of $G,$ we may assume that the action of $G$ on $S$ lifts to $G',$ and that $G$ and $G'$ have the same image in $\Aut(S).$ 
Then $T$ is $G'$-stable, so the action of $G$ on $T$ lifts to $G'.$ Conversely, if the action of $G$ on $T$ virtually lifts to $G',$ then, by (1), 
we may assume that the action of $G$ on $T$ lifts to $G',$ and that $G$ and $G'$ have the same image in $\Aut(T).$ In particular, any element of $G'$ acts on 
$\Aut(T)$ via an element of $\Aut(S).$ Since $\cup_{i=1}^n \pi_i(T) = S$ this element of $\Aut(S)$ is uniquely determined. 
The uniqueness implies that $G' \rightarrow \Aut(T)$ factors through a homomorphism $G' \rightarrow \Aut(S),$ 
which lifts the action of $G$ on $S.$
\end{proof}

\begin{para}
Let $X$ be a scheme, and $f:Y \rightarrow X$ a finite \'etale cover. We will say that $f$ is Galois
if any connected component of $Y$ is Galois over its image. If $\bar x$ is a geometric point of $X,$ then 
$Y$ corresponds to a finite set $S(Y)$ equipped with an action of $\pi_1(X,\bar x).$ Conversely, if $X$ is connected and 
$S$ is a finite set with an action of $\pi_1(X,\bar x),$ we denote by $Y(S)$ the corresponding finite \'etale cover 
of $X.$

In the next three lemmas we consider a map of schemes $g: X \rightarrow X'$ and a finite \'etale cover
$f': Y' \rightarrow X'$ equipped with an isomorphism $Y \iso Y'\times_{X'}X.$ If $X$ and $X'$ are connected, 
and $\bar x$ again denotes the geometric point of $X'$ induced by $\bar x,$ then such a $Y'$ exists if and only 
if the action of $\pi_1(X,\bar x)$ on $S(Y)$ lifts to $\pi_1(X',\bar x).$
\end{para}

\begin{lemma}\label{lem:coverlemI} If $f$ is Galois, then there is an finite \'etale $h: X'' \rightarrow X'$
such that $g$ factors through $X''$ and $Y'' = X''\times_{X'} Y' \rightarrow X''$ is Galois.
\end{lemma}
\begin{proof} We may assume that $X$ and $X'$ are connected. 
We apply Lemma \ref{lem:basicsetlemma}, (1), to $\pi_1(X,\bar x) \rightarrow \pi_1(X',\bar x)$ and $S = S(Y) = S(Y').$ 
Let $h: X'' \rightarrow X'$ be the finite \'etale map corresponding to the $\pi_1(X',\bar x)$-set $\pi_1(X',\bar x)/G''.$ 
Since this set has a $\pi_1(X,\bar x)$ fixed point, $g$ factors through $X''.$ By construction, 
the images of $\pi_1(X'',\bar x)$ and $\pi_1(X,\bar x)$ in $\Aut(S(Y')) = \Aut(S(Y))$ are equal. Denote this image by $H,$ 
and for $s \in S(Y)$ denote by $H_s$ the stabilizer.  
Since $Y\rightarrow X$ is Galois, $H_s$ is a normal subgroup which does not depend on $s,$ which is implies that $Y'' \rightarrow X''$ 
is Galois.
\end{proof}

\begin{para} Let $A$ be a finite ring. By an {\em $A$-local system} $\FF$ on $X$ we will mean an \'etale sheaf of $A$-modules which is locally isomorphic to the constant $A$-module $A^n$ for some $n.$ Such an $\FF$ is representable by a finite \'etale map
$Y(\FF) \rightarrow X.$ This is clear \'etale locally on $X,$ and follows from \'etale descent in general.
Isomorphism classes of $A$-local systems are in bijective correspondence with conjugacy classes of representations
$\pi_1(X) \rightarrow \GL_n(A).$

For $X \rightarrow X_1$ a map of schemes, and $\mathcal G$ an $A$-local system on $X_1,$ we will denote by $\mathcal G|_X$ the pullback of $\mathcal G$ to $X.$ 

\end{para}

\begin{lemma}\label{lem:coverlemII} Suppose that $f: Y = Y(\FF) \rightarrow X$ corresponds to an $A$-local system $\FF.$
Then there is a finite \'etale $h: X'' \rightarrow X'$ such that $g$ factors through $X''$
and $Y'' = X''\times_XY' \rightarrow X''$ represents an $A$-local system
$\FF'',$ with $\FF''|_X \iso \FF.$
\end{lemma}
\begin{proof}  Choose $X''$ as in the proof of Lemma \ref{lem:coverlemI}, using the construction in Lemma \ref{lem:basicsetlemma}, (1). Then the finite set $S(Y) = S(Y')$ naturally has the structure of a finite free $A$-module on which $\pi_1(X,\bar x)$ acts $A$-linearly. Since the images of $\pi_1(X'',\bar x)$ and $\pi_1(X,\bar x)$ in $\Aut(S(Y)$ are equal, $\pi_1(X'',\bar x)$ acts on $S(Y)$ $A$-linearly, which implies the statement of the lemma.
\end{proof}

\begin{para} For any integer $N \geq1$ we denote by $\mu_N$ the kernel of $\GG_m \overset N \rightarrow \GG_m.$ This is a finite flat group scheme on $\Z,$ and we denote by the same symbol its pullback to any scheme $X.$
This pullback is \'etale if any only if $X$ is a $\Z[1/N]$-scheme.

Suppose $X$ is a $\Z[1/N]$-scheme. 
For any $\Z/N\Z$ algebra $A,$ we again denote by $\mu_N$ or $A(1)$ the \'etale sheaf $\mu_N\otimes_{\Z/N\Z}A.$ 
For any non-negative integer $i$ we write $A(i) = A(1)^{\otimes i}.$ For $i$ negative we set 
$A(i)$ equal to the $A$-linear dual of $A(-i)$.
\end{para}

\begin{lemma}\label{lem:coverlemIII} Let $N \geq 1$ be an integer, $i,j$ integers, and $A$ a finite $\Z/N\Z$-algebra.
Suppose that $X$ is a $\Z[1/N]$-scheme and that $f: Y = Y(\FF) \rightarrow X$ corresponds to an $A$-local system $\FF,$ which is an extension of  $A(i)^r$ by $A(j)^s,$ for some positive integers $h,g.$

Then there is a finite \'etale map $X'' \rightarrow X'$ such that $g$ factors through $X''$
and the cover $Y'' = X''\times_XY' \rightarrow X''$ represents an $A$-local system
$\FF'',$  which is an extension of $A(i)^r$ by $A(j)^s,$ with $\FF''|_X \iso \FF$ as extensions.
\end{lemma}
\begin{proof} Replacing $\FF$ by $\FF(-j) : = \FF\otimes_AA(-j)$, we may suppose that $j=0.$
By Lemma \ref{lem:coverlemII}, we may assume $Y''$ represents an $A$-local system and
$g^*(\FF') \iso \FF$ as $A$-local systems. Let $\FF_1 \subset \FF$ denote the sub $A$-local system corresponding to
$A(j)^s,$ so that $\FF_1$ corresponds to an $A$-submodule $S(Y')_1\subset S(Y').$

Since the group $H$ acts trivially on $S_{Y',1},$ after replacing $X'$ by $X'',$
we may assume that $\FF'$ is an extension of $\FF'/A(i)^r$ by $A(j)^s,$ and
$g^*(\FF') \iso \FF$ is an isomorphism of extensions. A similar argument, applied to
$\FF/A(j)^s\otimes A(-i),$ shows that we may assume that $\FF'$ is an extension of $A(i)^r$ by 
$A(j)^s,$ with $\FF'|_X \iso \FF$ as extensions.
\end{proof}

\subsection{Essential dimension}
\begin{para} Let $K$ be a field, $X$ a $K$-scheme of finite type, and $f: Y \rightarrow X$ a finite \'etale cover.
The {\em essential dimension} \cite[\S 2]{BR} 
$\ed_K(Y/X)$ of $Y$ over $X$ is the smallest integer $e$ such that there exists
a finite type $K$-scheme $W$ of dimension $e,$ a dense open subscheme $U \subset X,$ and a map
$U \rightarrow W,$ such that $Y|_U$ is the pullback of a finite \'etale covering over $W.$
The {\em essential $p$-dimension} $\ed_K(Y/X;p)$ is defined as the minimum of $\ed_K(Y\times_XE/E)$ where
$E \rightarrow X$ runs over dominant, generically finite maps, which have degree prime to $p$ at all generic points of $X.$

Note that when $X$ is irreducible, the definition does not change if we consider only coverings with $E$ irreducible. 
Indeed, for any $E \rightarrow X,$ as above, one of the irreducible components of $E$ will have degree prime to $p$ over the 
generic point of $X.$

Suppose $X$ is connected. A {\em Galois closure} of $Y \rightarrow X$ is a union of Galois closures of the connected components  
of $Y.$ If $Y_i \rightarrow X$ $i=1, \dots r$ are finite \'etale, Galois and connected, then a {\em composite} of the $Y_i$ 
is a connected component of $Y_1\times_X\times \dots \times _X Y_r.$ Up to isomorphism, this does not depend on the choice of connected component. If we drop the assumption that $X$ is connected,  
we define the Galois closure and composite by making these constructions over each connected component of $X.$
\end{para}

\begin{lemma}\label{lem:indepGalois} Let $f:Y\rightarrow X$ be finite \'etale cover, and $\tilde Y$ a Galois closure for $Y.$
For any map of $K$-schemes $E \rightarrow X,$ we have 
$$ \ed_K(Y_E/E) = \ed_K(\tilde Y_E/E).$$ In particular, we have 
$$ \ed_K(Y/X;p) = \ed_K(\tilde Y/X;p).$$
\end{lemma}
\begin{proof} We may assume that $X$ is connected. Let $\bar x$ be a geometric point for $X,$ and $S$ a $\pi_1(X,\bar x)$-set 
with $Y = Y(S).$ Suppose first that $Y$ is connected. Let $N \subset \pi_1(X,\bar x)$ be the (normal) subgroup which fixes $S$ pointwise. If $s \in S,$ and $\pi_1(X,\bar x)_s$ is the stabilizer of $s,$ 
then $N = \cap_{i=1}^n g_i\pi_1(X,\bar x)_sg_i^{-1}$ 
for a finite collection of elements $g_1,\dots, g_n \in \pi_1(X,\bar x).$ Let $\tilde S$ denote the $\pi_1(X,\bar x)$-orbit of  
$(g_1\cdot s, \dots, g_n\cdot s) \in S^n.$ Then the stabilizer in $\pi_1(X,\bar x)$ of any point of $\tilde S$ is $N,$ 
so $\tilde Y = Y(\tilde S)$ is a Galois closure of $Y.$

Now we drop the assumption that $Y$ is connected, and let $Y_1, \dots, Y_r$ denote the connected components of $Y,$ 
with $Y_i$ corresponding to a subset $S_i \subset S$ on which $\pi_1(X,\bar x)$ acts transitively. 
The above construction gives subsets $\tilde S_i \subset S_i^n,$ 
(we may assume without loss of generality that $n$ does not depend on $i$) with $\tilde Y_i = Y(\tilde S_i)$ a Galois closure of $Y_i.$
If $\tilde S = \coprod_i \tilde S_i \subset S^n,$ then $\tilde Y = Y(\tilde S) = \coprod_i \tilde Y_i$ is a Galois closure for $Y.$

From the construction one sees that  $S,$ $T = \tilde S$ and $G = \pi_1(X,\bar x)$ satisfy the conditions in Lemma \ref{lem:basicsetlemma}, (2) 
(note that these conditions do not depend on the choice of $G'$). 
For any map of groups $H \rightarrow G,$ these conditions continue to hold if we instead view $S$ and $T$ as $H$-sets. 

Now let $E \rightarrow X$ be a map of $K$-schemes, $U \subset E$ 
a connected open subset, and $U \rightarrow E'$ a map of $K$-schemes. 
Suppose $U$ admits a geometric point $\bar y$ mapping to $\bar x.$ 
Applying the above remark with $H = \pi_1(U,\bar y),$ Lemma \ref{lem:basicsetlemma}, (2) implies that the 
action of $H$ on $S$ virtually lifts to $\pi_1(E',\bar y)$ if and only the action of $H$ on $T$ virtually lifts 
to $\pi_1(E',\bar y).$ As in the proof of Lemma \ref{lem:coverlemI}, this implies that $Y|_U$ arises by pullback from a cover of $E''$ for some finite \'etale $E'' \rightarrow E',$ 
if and only if the same condition holds for $\tilde Y|_U.$ Since $\bar x$ was an arbitary geometric point of $X,$ 
this implies $\ed_K(Y|_E/E) = \ed_K(\tilde Y|_E/E).$
\end{proof}

\begin{lemma}\label{lem:indepGaloisII} Let $Y_i \rightarrow X$ $i=1,\dots, r$ be connected Galois coverings of $K$-schemes, and 
$Y\rightarrow X$ a composite of the $Y_i.$ Then for any map of $K$-schemes $E\rightarrow X,$ we have 
$$ \ed_K(Y_E/E) = \ed_K((\coprod_i Y_i)_E/E).$$
 In particular, we have 
$$ \ed_K(Y/X;p) =  \ed_K((\coprod_i Y_i)/X; p).$$
\end{lemma} 
\begin{proof} Let $\bar x$ be a geometric point for $X,$ and let $S_i$ be a $\pi_1(X,\bar x)$-set with $Y_i = Y(S_i).$ 
Let $S = \coprod_i S_i.$ Then $Y$ corresponds to a transitive $\pi_1(X,\bar x)$ set $T \subset S_1\times\dots \times S_r$ 
which surjects onto each $S_i.$ Viewing $T \subset S^r,$ we see that $T$ satisfies the conditions of Lemma \ref{lem:basicsetlemma}. 
The lemma now follows as in the proof of Lemma \ref{lem:indepGalois}.
\end{proof}

\begin{lemma}\label{lem:equiedlem} Let $A$ be a finite ring, $K$ a field, $\FF$ an $A$-local system
on a connected $K$-scheme $X$ equipped with a geometric point $\bar x,$ and $\rho_{\FF}:\pi_1(X, \bar x) \rightarrow \GL_n(A)$ the representation corresponding to $\FF.$
Let $Y$ be the covering of $X$ corresponding to $\ker \rho_{\FF}.$

Then $Y$ is the composite of Galois closures of the connected components of $Y(\FF).$
In particular, $$\ed_K(Y'/X) = \ed_K(Y(\FF)/X).$$
For any prime $p$ we also have
$$\ed_K(Y/X; p) = \ed_K(Y(\FF)/X; p).$$
\end{lemma}
\begin{proof} Consider the action of $\pi_1(X,\bar x)$ on $A^n$ corresponding to $\FF.$
The connected components of $Y(\FF)$ correspond to the stabilizers $\pi_1(X)_s$ for $s \in A^n.$
A Galois closure of such a component corresponds to
$$ \cap_{g \in \pi_1(X)} g\pi_1(X)_sg^{-1} = \cap_{g \in \pi_1(X)} \pi_1(X)_{gs}.$$
Thus the composite of such Galois closures corresponds to $ \cap_s \pi_1(X)_s = \ker \rho_{\FF}.$

The lemma now follows from Lemmas \ref{lem:indepGalois} and \ref{lem:indepGaloisII}.

%
\end{proof}

%

\begin{lemma}\label{lem:redntoqbar} Let $K' \subset K$ be algebraically closed fields.
If $Y \rightarrow X$ is a finite \'etale covering of finite type $K'$-schemes then
\[\ed_{K'}(Y/X) = \ed_K(Y_K/X_K).\]
\end{lemma}
\begin{proof} Let $U_K \subset X_K$ be a dense open and $U_K \rightarrow W_K$ a morphism with
$\dim W_K = \ed_K(Y_K/X_K)$ such that $Y_K|_{U_K}$ arises from a finite cover $f':Y'_K \rightarrow W_K.$
Then $U_K,$ the finite cover $f'$ and the isomorphism $f^{\prime*} Y'_K \iso Y_K|_{U_K}$ are all defined over
some finitely generated $K'$-algebra $R \subset K.$ Specializing by a map $R \rightarrow K'$
produces the required data for the covering $Y \rightarrow X.$
\end{proof}

\begin{para}\label{defn:ed} It will be convenient to make the following definition. Suppose that $Y \rightarrow X$ is a finite \'etale covering
of finite type $R$-schemes, where $R$ is a domain of characteristic $0.$ Set
$ \ed(Y/X) = \ed_{\bar K}(Y/X)$ where $\bar K$ is any algebraically closed field containing $R,$ and similarly for $ \ed(Y/X; p)$
By Lemma \ref{lem:redntoqbar} this does not depend on the choice of $\bar K.$
\end{para}

\begin{lemma}\label{lem:brauer} Let $K$ be an algebraically closed field, and $Y \rightarrow X$ a finite Galois
covering of connected, finite type $K$-schemes with Galois group $G.$ Let $H \subset G$ be a central, cyclic subgroup of order $n$
with $\text{\rm char}(K)\nmid n.$ Then for any prime $p$ we have
$$ \ed(Y \to X; p) \geq \ed(Y/H\to X; p) $$
with equality if $p \nmid n.$
\end{lemma}
\begin{proof} To show $ \ed(Y/X; p) \geq \ed(Y/H\to X; p),$ after shrinking $X,$ we may assume there is
a map $f:X \rightarrow X'$ such that $Y = f^*Y'$ for a finite \'etale covering $Y' \rightarrow X',$ which may be
assumed to be connected and Galois by Lemma \ref{lem:coverlemI}. The Galois group of $Y'/X'$ is necessarily equal to $G,$ and we have $f^*(Y'/H) \iso Y/H.$

For the converse inequality when $p \nmid n,$ we may assume there is a map $f:X \rightarrow X'$ such that $Y/H = f^*Y'$
for a finite \'etale covering $Y' \rightarrow X',$ which we may again assume is connected and Galois with group $G/H.$
The image, $c,$ of $Y' \rightarrow X'$ under 
$$ H^1(X', G/H) \rightarrow H^2(X', H) \iso H^2(X', \mu_n) $$
is the obstruction to lifting $Y' \rightarrow X'$ to a $G$-covering. 
Here, for the final isomorphism, we are using that $K$ is algebraically closed of characteristic prime to $n.$ Viewing $c$ as a Brauer class, we see that it has order dividing $n.$ This implies that (after perhaps shrinking $X$ further) there is an \'etale covering 
$X_1' \rightarrow X'$ of order dividing a power $n$ such that $c|_{X_1}$ is trivial \cite[Lemma 4.17]{FarbDennis}. In particular, $X_1' \rightarrow X'$ has order prime to $p.$ Replacing $Y \rightarrow X \rightarrow X'$ by their pullbacks to $X_1',$ we may assume that $c = 0,$ and that $Y' = Y''/H$ for some Galois covering $Y'' \rightarrow X'$ with group $G.$

The difference between the $G$-coverings $Y \rightarrow X$ and $f^*Y'' \rightarrow X$ is measured by a class in
$H^1(X,H).$ After replacing $X$ by the $H$-covering corresponding to this class, we may assume that this class is trivial,
and so $Y \iso f^*Y''.$ This shows that $ \ed(Y/X; p) \leq \ed(Y/H\to X; p),$
\end{proof}

\section{Essential dimension and moduli of abelian varieties}\label{sec:moduli}
\numberwithin{equation}{subsection}
\subsection{Ordinary finite flat group schemes}\label{subsec:ffgs}
In this subsection, we fix a prime $p,$ and we consider a complete discrete valuation ring $V$ of characteristic $0,$
with perfect residue field $k$ of characteristic $p,$ and a uniformizer $\pi \in V.$

By a {\it finite flat group scheme} on a $\Z_p$-scheme  $X$ we will always mean a finite flat, commutative, group scheme on $X$ of $p$-power order.  A finite flat group scheme on $X$ is called {\em ordinary} if \'etale locally on $X,$ it is an extension of a constant group scheme $\oplus_{i \in I} \Z/p^{n_i}\Z$ by a group scheme of the form $\oplus_{j \in J} \mu_{p^{m_j}}$ for integers $n_i,m_j \geq 1.$ In this subsection we study the classification of these extensions.

\begin{para} \label{para:ffgssetup}
Now let $\tilde X = \Spec A$ be an affine $\Z_p$-scheme, and set $X = \tilde X\otimes \Q.$
Let $n \geq 1,$ and consider the exact sequence of sheaves
$$ 1 \rightarrow \mu_{p^n} \rightarrow \GG_m \overset{p^n}\rightarrow  \GG_m \rightarrow 1$$
in the flat topology of $\tilde X.$ Taking flat cohomology of this sequence and its restriction to $X$
we obtain a commutative diagram with exact rows
$$\xymatrix{
1 \ar[r] & A^\times/(A^\times)^{p^n} \ar[r]\ar[d] & H^1(\tilde X, \mu_{p^n}) \ar[r]\ar[d] & H^1(\tilde X, \GG_m) \ar[d] \\
1 \ar[r] & A[1/p]^\times/(A[1/p]^\times)^{p^n}\ar[r] & H^1(X, \mu_{p^n}) \ar[r] & H^1(X, \GG_m)
 }$$

The group $H^1(\tilde X, \GG_m)$ classifies line bundles on $\tilde X.$ Hence, if $A$ is local it vanishes,
and this can be used to classify extensions of $\Z/p^n\Z$ by  $\mu_{p^n}$ as finite flat group schemes. We have
$$ \Ext^1_{\tilde X}(\Z/p^n\Z, \mu_{p^n}) \iso  H^1(\tilde X, \mu_{p^n}) \iso A^\times/(A^\times)^{p^n}.$$
Here and below, the group on the left denotes extensions as sheaves of $\Z/p^n\Z$-modules.

Similarly, we can classify extensions of $\Q_p/\Z_p$ by $\mu_{p^\infty} = \lim_n \mu_{p^n}$
as $p$-divisible groups. If $\mathcal E$ is such an extension, then $\mathcal E[p^n]$ is an extension 
of $\Z/p^n\Z$ by $\mu_{p^n}$ and we have 
$$ \hat \theta_A: \Ext^1_{\tilde X}(\Q_p/\Z_p,\mu_{p^\infty}) \iso \ilim_n A^\times/(A^\times)^{p^n}.$$

If $A$ is complete and local with residue field $k,$ then the right hand side may be identified with 
$A^{\times,1} \subset A^\times,$ the subgroup of units which map to $1$ in $k^\times.$ Thus we have 
$$ \hat \theta_A: \Ext^1_{\tilde X}(\Q_p/\Z_p,\mu_{p^\infty}) \iso A^{\times,1}.$$
\end{para}

 \begin{para}
 For the rest of this subsection we assume that $A = V\lps x_1, \dots, x_n \rps.$
 Then $H^1(X, \GG_m) = 0$ by \cite{SGA2}*{XI, Thm 3.13}, and we have a commutative diagram
$$\xymatrix{
A^\times/(A^\times)^{p^n} \ar[r]^\sim\ar[d] & H^1(\tilde X, \mu_{p^n}) \ar[d] \\
A[1/p]^\times/(A[1/p]^\times)^{p^n} \ar[r]^{\phantom{ttttttt}\sim} & H^1(X, \mu_{p^n})
 }$$

\end{para}

\begin{para}
We call an element of $\Ext^1_X(\Z/p^n\Z, \mu_{p^n}) = H^1(X,\mu_{p^n})$ {\em syntomic} if it arises
from an element of $A^\times,$ or equivalently from a class in $\Ext^1_{\tilde X}(\Z/p^n\Z, \mu_{p^n}),$
and we denote by $\Ext^{1,\syn}_X(\Z/p^n\Z, \mu_{p^n}) \subset \Ext^1_X(\Z/p^n\Z, \mu_{p^n})$ the subgroup of syntomic elements.
\end{para}

\begin{lemma}\label{lem:synI} A syntomic class in $\Ext^1_X(\Z/p^n\Z, \mu_{p^n})$ arises from a unique
class in $\Ext^1_{\tilde X}(\Z/p^n\Z, \mu_{p^n}).$
\end{lemma}
\begin{proof} If $a \in A^\times$ is a $p^{n\text{\rm th}}$-power in $A[1/p]$ then it is a $p^{n\text{\rm th}}$ power in $A,$ as $A$ is normal.
Hence the map $A^\times/(A^\times)^{p^n} \rightarrow A[1/p]^\times/(A[1/p]^\times)^{p^n}$ is injective, and the lemma follows from the description of $\Ext^1$'s above.
\end{proof}

\begin{lemma}\label{lem:synII} Let $B = V\lps y_1, \dots , y_s \rps$  for some integer $s \geq 0,$ and
$$ f: \tilde X \rightarrow \tilde Y = \Spec B $$ a local flat map of complete local $V$-algebras.
Suppose that $c \in H^1(Y,\mu_{p^n}),$ where $Y = \Spec B[1/p],$ and that $f^*(c) \in H^1(X,\mu_{p^n})$ is syntomic. Then $c$ is syntomic.
\end{lemma}
\begin{proof} Let $b \in B[1/p]^\times$ be an element giving rise to $c.$ Since $B$ is a unique factorization domain we may
write $b = b_0\pi^i$ with $b_0 \in B^\times$ and $i \in \Z.$ Since $f^*(c)$ is syntomic, we may write
$\pi^i = a_0a^{p^n}$ with $a_0 \in A^\times$ and $a \in A[1/p]^\times.$
Comparing the images of both sides in the group of divisors on $A,$ one sees that $p^n|i.$ So $c$ arises from $b_0.$
\end{proof}

\begin{para}\label{para:defntheta} Let $\gm_A$ be the maximal ideal of $A,$ and $\bar \gm_A$ its image in $A/\pi A.$
The natural map
$$ \bar\gm_A/\bar \gm_A^2 \overset {a \mapsto 1+ a} \rightarrow k^\times\backslash (A/(\pi,\gm_A^2))^\times$$
is a bijection; both sides are $k$-vector spaces spanned by $x_1,\dots, x_n.$
We denote by $\theta_A$ the composite
$$ \theta_A: \Ext_X^{1,\syn} (\Z/p\Z, \mu_p) \iso A^\times/(A^\times)^p \rightarrow k^\times\backslash (A/(\pi,\gm_A^2))^\times \iso \bar\gm_A/\bar \gm_A^2.$$
Here we have used Lemma \ref{lem:synI} to identify $ \Ext_X^{1,\syn} (\Z/p\Z, \mu_p)$ and
$ \Ext_{\tilde X}^1 (\Z/p\Z, \mu_p).$
\end{para}

\begin{lemma}\label{lem:synIII} With the notation of Lemma \ref{lem:synII}, suppose that
$$L \subset  \Ext_Y^{1,\syn} (\Z/p\Z, \mu_p)$$ is a subset such that the $k$-span of $\theta_A(f^*(L))$ is
$\bar\gm_A/\bar \gm_A^2.$ Then $f$ is an isomorphism.
\end{lemma}
\begin{proof} By functoriality of the association $A \mapsto \theta_A,$ we have
$\theta_A(f^*(L)) = f^*(\theta_B(L)).$ Hence the $k$-span of $\theta_A(f^*(L))$ is
contained in image of $\bar \gm_B/\bar \gm_B^2.$ It follows that
$\bar \gm_B/\bar \gm_B^2$ surjects onto $\bar \gm_A/\bar \gm_A^2.$
Since $A$ and $B$ are complete local $V$-algebras, this implies that $B,$
which is a subring of $A,$ surjects onto $A.$ Hence $f$ is an isomorphism.
\end{proof}

%
%

\subsection{Monodromy of $p$-torsion in an abelian scheme}
We now use the results of the previous section to obtain results about the essential dimension of covers of
of the moduli space of abelian varieties.

\begin{para}\label{para:univdef} Recall that an abelian scheme $\A$ over a $\Z_p$-scheme is called ordinary if the group scheme
$\A[p^n]$ is ordinary for all $n \geq 1.$ This is equivalent to requiring the condition for $n=1.$

Let $k$ be an algebraically closed field of characteristic $p > 0,$ and let $V$ be a complete discrete valuation ring with residue field $k,$ so that $W(k) \subset V.$ Let $\A_0$ be an abelian scheme over $k$ of dimension $g.$ We assume that $\A_0$ is ordinary. Since $k$ is algebraically closed, this implies that $\A_0[p^\infty]$ is isomorphic to
$(\Q_p/\Z_p)^g \oplus \mu_{p^\infty}^g.$

Consider the functor $D_{\A_0}$ on the category of Artinian $V$-algebras $C$ with residue field $k,$
which attaches to $C$ the set of isomorphism classes of deformations of $\A_0$
to an abelian scheme over $C.$ Recall \cite[\S 2]{Katz:Serre-Tate} that $D_{\A_0}$ is equivalent to the functor
which attaches to $C$ the set of isomorphism classes of deformations of $\A_0[p^\infty],$ and that $D_{\A_0}$
is pro-representable
by a formally smooth $V$-algebra $R$ of dimension $g^2,$ called the universal deformation $V$-algebra of $\A_0.$

We denote by $\A_R$ the universal (formal) abelian scheme over $R.$
Note that although $\A_R$ is only a formal scheme over $R,$ the torsion group schemes $\A_R[p^n]$ are
finite over $R,$ and so can be regarded as genuine $R$-schemes.

Since $\A_0$ is ordinary the $p$-divisible group $\A_R[p^\infty] = \lim_n \A_R[p^n]$ is an extension of
$(\Q_p/\Z_p)^g$ by $\mu_{p^\infty}^g.$ Hence $\A_R[p]$ is an extension of $(\Z/p\Z)^g$ by $\mu_p^g.$
This extension class is given by a $g\times g$ matrix of classes $(c_{i,j})$ with $c_{i,j} \in \Ext_R^1(\Z/p\Z, \mu_p).$
\end{para}

\begin{lemma}\label{lem:nondegen} With the notation of \S \ref{subsec:ffgs}, the elements
$\theta_R(\{c_{i,j}\}_{i,j})$ span $\bar \gm_R/\bar \gm_R^2.$
\end{lemma}
\begin{proof} Consider the isomorphism
$$ \hat \theta_R: \Ext_R^1(\Q_p/\Z_p, \mu_{p^\infty}) \iso R^{\times,1} $$
introduced in \S \ref{subsec:ffgs}. The universal extension of $p$-divisible groups over $R$ gives
rise to a $g\times g$ matrix of elements $(\hat c_{i,j}) \in \Ext_R^1(\Q_p/\Z_p, \mu_{p^\infty})$
which reduce to $(c_{i,j}).$

Let $L \subset \bar \gm_R/\bar \gm_R^2$ be the $k$-span of the images of the elements $\hat\theta_R(\hat c_{i,j})-1,$
or equivalently, the elements $\theta_R(c_{i,j})-1,$ and set $R' = k \oplus L \subset R/(\pi,\gm_R^2).$
Using the isomorphism $\hat \theta_{R'},$ one sees that $\A_R[p^\infty]|_{R/(\pi,\gm_R^2)}$ is defined over $R'.$
If $L \subsetneq \bar \gm_R/\bar \gm_R^2,$
then there exists a surjective map $R/(\pi,\gm_R^2) \rightarrow k[x]/x^2$ which sends $L$ to zero.
Specializing $\A_R[p^\infty]|_{R/(\pi,\gm_R^2)}$ by this map
induces the trivial deformation of $\A_0[p^\infty]$ (that is the split extension
of $(\Q_p/\Z_p)^g$ by $\mu_{p^\infty}$) over $\Spec k[x]/x^2.$
This contradicts the fact that $R$ pro-represents $D_{\A_0}.$
Hence $L = \bar \gm_R/\bar \gm_R^2,$ which proves the lemma.
\end{proof}

\begin{para}
Let $A$ be a quotient of $R$ which is formally smooth over $V.$ That is, $A$ is isomorphic as a complete $V$-algebra to $V\lps x_1, \dots x_n \rps.$ As in \S \ref{subsec:ffgs}, we set $X = \Spec A[1/p]$ and $\tilde X = \Spec A.$
\end{para}

\begin{lemma}\label{lem:essdimI} Let $B = V\lps y_1, \dots y_s \rps$  for some integer $s \geq 0,$ and let
$$ f: \tilde X \rightarrow \tilde Y = \Spec B $$ be a local flat map of complete local $V$-algebras. Set
$Y = \Spec B[1/p].$  Suppose that $k$ is algebraically closed, and that there exists an $\F_p$-local system $\L$ on $Y$
which is an extension of $(\Z/p\Z)^g$ by $\mu_p^g$ such that $f^*\L \iso \A_R[p]|_X$ as extensions of $\F_p$-local systems.
Then $f$ is an isomorphism.
\end{lemma}
\begin{proof}  Using the notation of \ref{para:univdef}, we have $\theta_R(\{c_{i,j}\}_{i,j})$ spans $\bar \gm_R/\bar \gm_R^2$ by Lemma \ref{lem:nondegen}. In particular, if we again denote by $c_{i,j}$ the restrictions of these classes to $A,$ then $\theta_A(\{c_{i,j}\}_{i,j})$ spans $\bar \gm_A/\bar \gm_A^2.$

Now by Lemma \ref{lem:synII} the $g^2$ extension classes defining $\L$ are syntomic. So $\L$ arises
from an extension of $(\Z/p\Z)^g$ by $\mu_p^g$ as finite flat group schemes over $\tilde Y.$
If we denote by $(d_{i,j})$ the corresponding $g\times g$ matrix of elements of $\Ext_{\tilde Y}^1(\Z/p\Z, \mu_p),$
then Lemma \ref{lem:synI}, together with the fact that $f^*\L \iso \A_R[p]|_X$ implies that $f^*(d_{i,j}) = c_{i,j}.$
It follows that the elements $\theta_A(f^*(\{d_{i,j}\}))_{i,j}$ span $\bar \gm_A/\bar \gm_A^2,$ which implies that $f$ is an isomorphism by Lemma \ref{lem:synIII}.
\end{proof}

%

\begin{para} Fix an integer $g \geq 1$, a prime $p\geq 2$, and a positive integer $N \geq 2$ coprime to $p.$
Consider the ring $\Z[\zeta_N][1/N],$ where $\zeta_N$ is a primitive $N^{\text{\rm th}}$ root of $1.$
Using the isomorphism $\Z/N\Z \underset {1 \mapsto \zeta_N}{\iso}\mu_N,$ for any
$\Z[\zeta_N][1/N]$-scheme $T,$ and any principally polarized abelian scheme $A$ over $T,$ the $N$-torsion scheme $A[N]$ is equipped with the (alternating) Weil pairing
$$ A[N] \times A[N] \rightarrow \Z/N\Z. $$
We denote by $\A_{g,N}$
the $\Z[\zeta_N][1/N]$-scheme which is the coarse moduli space of principally polarized abelian schemes $A$ of dimension $g$
equipped with a symplectic basis of $A[N].$
When $N \geq 3,$ this is a fine moduli space which is smooth over $\Z[\zeta_N][1/N].$ For a $\Z[\zeta_N][1/N]$-algebra $B,$ we denote by
$\A_{g,N/B}$ the base change of $\A_{g,N}$ to $B.$ If no confusion is likely to result we sometimes denote this base change simply by $\A_{g,N}.$

Suppose that $N \geq 3,$ and let $\A \rightarrow \A_{g,N}$ be the universal abelian scheme. The $p$-torsion subgroup $\A[p] \subset \A$ is a finite flat group scheme over $\A_{g,N}$ which is \'etale over $\Z[\zeta_N][1/Np].$ Let $x \in \A_{g,N}$ be a point with residue field $\kappa(x)$ of characteristic $p,$ and $\A_x$ the corresponding abelian variety over $\kappa(x).$
The set of points $x$ such that $\A_x$ is ordinary  is an open subscheme $\A_{g,N}^\ord \subset \A_{g,N}\otimes \F_p$. 
For any $N,$ we denote by $\A_{g,N}^\ord \subset \A_{g,N}\otimes \F_p$ the image of $\A_{g,NN'}$ for any 
$N' \geq 3$ coprime to $N$ and $p.$ 

We now denote by $k$ a perfect field of characteristic $p,$ and $K/W[1/p]$ a finite extension with ring of integers $\O_K$ and
uniformizer $\pi.$ We assume that $K$ is equipped with a choice of primitive $N^{\text{\rm th}}$ root of $1,$ $\zeta_N \in K.$
We remind the reader regarding the convention for the definition of $\ed$ and $\ed(\,\cdot\, ; p)$ introduced in
\ref{defn:ed}.
\end{para}

 \begin{thm}\label{thm:essdimAg}
 Let $g\geq 1$ and let $p$ be any prime.  Let $N \geq 3$ and coprime to $p,$ and let $\ZZ \subset \A_{g,N/\O_K}$ be a locally closed subscheme which is equidimensional and smooth over $\O_K,$ and whose special fiber, $\ZZ_k$, meets the ordinary locus 
 $\A_{g,N}^{\ord} \subset \A_{g,N/k}.$ 
Then $$ \ed(\A[p]|_{\ZZ_K}/\ZZ_K; p) = \dim \ZZ_K. $$
\end{thm}
\begin{proof}  It suffices to prove the theorem when $k$ is algebraically closed which we assume from now on.
Moreover, since $K$ is an arbitrary finite extension of $W[1/p],$ it is enough to show that
$ \ed_K(\A[p]|_{\ZZ_K}/\ZZ_K;p) = \dim \ZZ_K. $ We may replace $\ZZ$ by a component whose special fibre meets the ordinary locus, and assume that $\ZZ_K$ and $\ZZ_k$ are geometrically connected.

Suppose that $ \ed(\A[p]|_{\ZZ_K}/\ZZ_K; p) < \dim \ZZ_K.$ Then there exists a dominant, generically finite map
$U_K \rightarrow \ZZ_K$ of degree prime to $p$ at the generic points of $U_K,$
and a map $h: U_K \rightarrow Y_K$ to a finite type $K$-scheme $Y_K$ with
$\dim Y_K < \dim \ZZ_K,$ such that $\A[p]|_{U_K}$ arises as the pullback of a finite \'etale covering of $Y_K.$
We may assume that $Y_K$ is the scheme-theoretic image of $U_K$ under $h.$ Next, after replacing both $U_K$ and $Y_K$ by dense affine opens, we may assume that both these schemes are affine corresponding to $K$-algebras $B_K$ and $C_K$ respectively,
and that $U_K \rightarrow Y_K$ is flat.

Let $\tilde \ZZ$ be the normalization of $\ZZ$ in $U_K.$ Let $\frak p$ be the generic point of $\ZZ_k,$ and
$\frak q_1, \dots, \frak q_m$ the primes of $\tilde \ZZ$ over $\frak p.$ Since the degree of $\tilde \ZZ \rightarrow \ZZ$ over $\frak p$
is prime to $p,$ for some $i$ the ramification degree $e(\frak q_i/\frak p)$ and the degree of the residue field extension 
$\kappa(\frak q_i)/\kappa(\frak p)$ are prime to $p.$ In particular, the residue field extension is seperable. 
By Abhyankar's Lemma, it follows that, after replacing $K$ by a finite extension, we may assume that $e(\frak q_i/\frak p) = 1$ for some $i,$ and that $\tilde \ZZ \rightarrow \ZZ$ is \'etale at $\frak q_i.$
Shrinking $U_K$ further if necessary, we may assume that there is an affine open $\Spec B = U \subset \tilde \ZZ$ such that
$U_k \rightarrow \ZZ_k$ has dense image,  $U\otimes K = U_K,$ and $U \rightarrow \ZZ$ is \'etale.
In particular, $U$ is smooth over $\O_K.$

Now choose a finitely generated $\O_K$-subalgebra $C \subset C_K \cap B$ such that $C\otimes K = C_K.$ This is possible as $C_K$ is finitely generated over $K.$ Then $h$ extends to a map $h: U \rightarrow Y = \Spec C.$ Let $J \supset (p)$ be an ideal of $C,$
and $Y_J \rightarrow Y$ the blow up of $J.$ Denote by $U_J$ the proper transform of $U$ by this blow up.
That is, $U_J$ is the closure of $U_K$ in $U\times Y_J.$
By the Raynaud-Gruson flattening theorem, \cite[Thm~5.2.2]{RG}, we can choose $J$ so that $U_J \rightarrow Y_J$ is flat. Since $U$ is normal, the map $U_J \rightarrow U$ is an isomorphism over the generic points of $U\otimes k.$ Hence, after replacing $Y$ by an affine open in $Y_J,$ and shrinking $U,$ we may assume that
$U \rightarrow Y$ is flat.

Shrinking $U$ further, we may assume that the special fiber $U_k$ maps to the ordinary locus of $\A_{g,N}.$
Now let $\widehat B$ and $\widehat C,$ denote the $p$-adic completions of $B$ and $C$ respectively,
and set $\widehat U = \Spec \widehat B$ and $\widehat Y = \Spec \widehat C.$
\footnote{Although it would in some sense be more natural to work with formal schemes here, we stay in the world of affine schemes, so as to be able to apply the results proved in \S 1, and to deal with generic fibers without resorting to $p$-adic analytic spaces.}
Since $\A[p]|_{\widehat U}$ is ordinary, there is a finite \'etale covering $\widehat U' = \Spec \widehat B' \rightarrow \widehat U$ such that $\A[p]|_{\widehat U'}$
is an extension of $(\Z/p\Z)^g$ by $\mu_p^g.$
Hence by Lemmas \ref{lem:coverlemII} and \ref{lem:coverlemIII},
 $\widehat U'_K \rightarrow \widehat Y_K$ factors through a finite \'etale map
 $\widehat Y'_K \rightarrow \widehat Y_K$ such that $\A[p]|_{\widehat U'_K}$ is the pullback of
 an extension $\FF'$ of $(\Z/p\Z)^g$ by $\mu_p^g$ on $\widehat Y'_K.$ As $\widehat U'$ is normal,
 we may assume $\widehat Y'_K$ is normal.

Let $\widehat Y' = \Spec \widehat C'$ be the normalization of $\widehat Y$ in $\widehat Y'_K.$ As $\widehat U'$ is normal, we have
 $$ \widehat U' \rightarrow \widehat Y' \rightarrow \widehat Y.$$
 As $\widehat Y'$ is normal, $\widehat U' \rightarrow \widehat Y'$ is flat over the generic points of $\widehat Y'_k.$
 Hence, there exists $f_0 \in \widehat C'/\pi \widehat C'$ which is nowhere nilpotent on $\widehat Y'_k,$ and such that
 $\widehat U'_k \rightarrow \widehat Y'_k$ is flat over the complement of the support of the ideal $(f_0).$
 Now let $f \in \widehat C'$ be a lift of $f_0,$ and let $\widehat C'' = \widehat {C'[1/f]}$ and $\widehat B'' = \widehat {B'[1/f]},$
 the $p$-adic completions
 \footnote{$ \widehat C''$ corresponds to a formal affine open in the formal scheme $\Spf \widehat C'.$}
 of $C'[1/f]$ and $B'[1/f].$
 Let $\widehat U'' = \Spec B''$ and $\widehat Y'' = \Spec \widehat C''.$
 Then $\widehat U''$ is flat over $\widehat Y''$ by \cite[IV, 11.3.10.1]{EGA}.
 Moreover, since $\widehat U \rightarrow \widehat Y$ is flat, the generic points of $\widehat U'_k$ map to generic points of
 $\widehat Y'_k.$ So the image of $\widehat U''_k$ is dense in $\widehat U'_k,$ and in particular $\widehat U''(k)$ is non-empty.

Now choose a point $x \in \widehat U''(k),$ and denote by $y \in \widehat Y''(k)$ its image.
We write $\O_{\widehat U'',x}$ and $\O_{\widehat Y'',y}$ for the complete local rings at $x$ and $y.$
Since the maps
$$ \widehat U'' \rightarrow \widehat U' \rightarrow \widehat U \rightarrow \ZZ $$
are formally \'etale,  $\O_{\widehat U'',x}$ is naturally isomorphic to the complete local ring at the image of $x$ in $\ZZ.$
Let $R$ be the universal deformation $\O_K$-algebra of the abelian scheme $\A_x.$ Then $\O_{\widehat U'',x}$ is naturally a quotient of $R.$ The map $\O_{\widehat Y'',y} \rightarrow \O_{\widehat U'',x}$ satisfies the conditions of Lemma \ref{lem:essdimI}, (cf.~\cite[IV, 17.5.3]{EGA}) and it follows that this map is an isomorphism. In particular, this implies that
$$ \dim Y_K = \dim \O_{\widehat Y'',y} -1 = \dim \O_{\widehat U'',x} -1 = \dim \ZZ_K$$
which contradicts our initial assumption.
\end{proof}

\begin{cor}\label{cor:essdimAg} Let $g\geq 1,$ be an integer, $p$ any prime, and $N \geq 1$ an integer coprime to $p.$
Let $\ZZ \subset \A_{g,N/\O_K}$ be a locally closed subscheme which is equidimensional and smooth over $\O_K,$ and whose special fiber, $\ZZ_k$, meets the ordinary locus  $\A_{g,N}^{\ord} \subset \A_{g,N/k}.$

If $N=1,2$ we also assume the following condition: For any generic point $\eta \in \ZZ_k,$ and $\bar\eta$ the spectrum of an algebraic closure of $\kappa(\eta),$ the abelian variety  $\A_{\bar\eta}$ over $\bar\eta,$ has automorphism group equal to $\{\pm 1\}.$
Then
$$ \ed(\A_{g,pN}|_{\ZZ_K}/\ZZ_K; p) = \dim \ZZ_K .$$
\end{cor}
\begin{proof} If $N \geq 3,$ the corollary follows from Theorem \ref{thm:essdimAg} and Lemma \ref{lem:equiedlem}.

Suppose $N=1$ or $2.$ Let $N' \geq 3$ be an integer coprime to $pN.$ We may assume that $K$ is equipped with
a primitive $pN^{\prime\text{\rm th}}$ root of $1,$ $\zeta_{pN'}.$
The map of $K$-schemes $\A_{g,pN} \rightarrow \A_{g,N},$ is a covering with group $\Sp_{2g}(\Bbb F_p)/\{\pm 1\}.$
Consider the maps
$$\A_{g,pNN'} \rightarrow  \A_{g,pN}\times_{\A_{g,N}}\A_{g,NN'} = : \A'_{g,pNN'} \rightarrow \A_{g,NN'}.$$
Then $\A'_{g,pNN'} \rightarrow \A_{g,NN'}$ again corresponds to a $\Sp_{2g}(\Bbb F_p)/\{\pm 1\}$
covering, and $\A_{g,pNN'} \rightarrow \A_{g,NN'}$ corresponds to a $\Sp_{2g}(\Bbb F_p)$ covering.
When $p=2$ these two coverings coincide.

Let $\ZZ_{N'} = \ZZ \times_{\A_{g,N}}\A_{g,NN'}.$ Our assumption on the automorphisms of $\A_{\bar\eta}$
implies that at the generic points of $\ZZ_k,$ the map $\ZZ_{N'} \rightarrow \ZZ$ is \'etale.
Thus, after replacing $\ZZ$ by a fibrewise dense open, we may assume that
$\ZZ_{N'}$ is smooth over $\O_K.$ The observations of the previous paragraph, Lemma \ref{lem:brauer} when $p >2,$
Lemma \ref{lem:equiedlem} and Theorem \ref{thm:essdimAg} imply that
we have
 \begin{multline} \ed(\A_{g,pN}|_{\ZZ_K}/\ZZ_K; p) \geq \ed(\A'_{g,pNN'}|_{\ZZ_{N',K}}/\ZZ_{N',K}; p) \\
=  \ed(\A_{g,pNN'}|_{\ZZ_{N',K}}/\ZZ_{N',K}; p) = \ed(\A[p]|_{\ZZ_{N',K}}/\ZZ_{N',K}; p) = \dim \ZZ_K.
\end{multline}

\end{proof}

\begin{cor}\label{cor:essdimAgII} Let $g, n\geq 2$ and $N$ a positive integer coprime to $n.$
Consider the finite \'etale map of $\Q(\zeta_{nN})$-schemes $\A_{g,nN} \rightarrow \A_{g,N}.$
Then for any $p\mid n$ we have
$$ \ed(\A_{g,nN}/\A_{g,N}; p) =  \dim \A_g = \binom{g+1}{2}.$$
\end{cor}
\begin{proof} A fortiori it suffices to consider the case when $n=p$ is prime,
which is a special case of \ref{cor:essdimAg}.
\end{proof}

\subsection{Moduli spaces of curves} Using the Torelli theorem one can use  Theorem \ref{thm:essdimAg}
 to deduce the essential $p$-dimension of certain coverings of families of curves.

\begin{para}
Let $g \geq 2,$ and let $\M_g$ denote the coarse moduli space of smooth, proper, genus $g$ curves.
For any integer $n,$ let $\M_g[n]$ denote the $\Z[\zeta_n][1/n]$-scheme which is the coarse moduli space of pairs $(C,{\mathcal B})$ consisting of a proper smooth curve $C$ of genus $g,$ together with a choice $\mathcal B$ of symplectic basis for $J(C)[n],$ where $J(C)$ denotes the Jacobian of $C.$ For $n \geq 3$ this is a fine moduli space which is smooth over $\Z[\zeta_n][1/n]$ \cite{DeligneMumford}.

\end{para}

\begin{thm}\label{thm:genusg} Let $g,n \geq 2$, and let $p$ be any prime dividing $n.$  Then
    \[
        \ed(\M_g[n]/\M_g; p)=\dim \M_g=3g-3.
    \]
\end{thm}
\begin{proof}
Let $N \geq 3$ be an integer.
There is a natural map of $\Z[\zeta_N][1/N]$-schemes
$\varpi: \M_g[N] \rightarrow A_{g,N}$ taking a curve to its Jacobian. For $(C,\mathcal B)$ in $\M_g[N]$
the pairs $(J(C), \mathcal B)$ and $(J(C), -\mathcal B)$ are isomorphic via $-1$ on $J(C).$ Thus
$(C,\mathcal B) \mapsto (C,-\mathcal B)$ is an involution $\Sigma$ of $\M_g[N],$ which is non-trivial, unless $g=2.$
We denote by $\M_g[N]'$ the quotient of $\M_g[N]$ by this involution. When $g \geq 3,$ the fixed points of $\Sigma$
in any fibre of $\M_g[N]$ over $\Spec \Z[\zeta_N][1/N],$ are contained in a proper closed subset. Thus
$\M_g[N]'$ is generically smooth over every point of  $\Spec \Z[\zeta_N][1/N].$

By \cite{OortSteenbrink} 1.11, 2.7, 2.8, the map of $\Z[\zeta_N][1/N]$-schemes $\M_g[N] \rightarrow \A_{g,N}$ induces a map $\M_g[N]' \rightarrow \A_{g,N}$ which is injective, an immersion if $g=2$ and an immersion outside the hyperelliptic locus if $g \geq 3.$

It suffices to prove the theorem with $n$ replaced by the prime factor $p.$
Let $N \geq 3$ be coprime to $p,$ and let $\M_g[pN]'' = \M_{g}[p]\times_{\M_g}{\M_g[N]}'.$
It is enough to show that $\ed(\M_g[pN]''/\M_g[N]'; p) = 3g-3.$
If $p=2$ then $\M_g[pN]'' = \M_g[pN]'$ as coverings of $\M_g[N]',$ and if $p \geq 3$ we
have natural degree $2$ maps of  coverings of $\M_g[N]'$
$$ \M_g[pN]'' \leftarrow \M_g[pN] \rightarrow \M_g[pN]'.$$
Thus, using Lemma \ref{lem:brauer} when $p \geq 3,$ it suffices to show that
$$\ed(\M_g[pN]'/\M_g[N]';p) = 3g-3.$$

Since $\M_g[pN]' = \A_{g,pN}|_{\M_g[N]'},$ for example by comparing the degrees of these coverings, and
$\M_g[N]'$ meets the ordinary locus in $\A_{g,N}\otimes \F_p$ \cite[2.3]{FabervanderGeer},
the theorem now follows from Theorem \ref{thm:essdimAg}.
\end{proof}

\begin{para} We now prove the analogue of Theorem~\ref{thm:genusg} for the moduli space of hyperelliptic curves. Let $S$ be a $\Z[1/2]$-scheme. Recall that a {\em hyperelliptic curve} over $S$ is a smooth proper curve $C/S$ of genus $g \geq 1,$ equipped with an involution $\sigma$ such that $P = C/\langle \sigma \rangle$ has genus $0.$ Let $\H_g$ denote the coarse moduli space of genus $g$ hyperelliptic curves over $\Z.$ It is classical (and not hard to see) that over $\Z[1/2]$ one has
\[\H_g \cong \M_{0,2g+2}/S_{2g+2}\] where $\M_{0,2g+2}$ is the moduli space of genus $0$
curves with $2g+2$ ordered marked points, and $S_{2g+2}$ is the symmetric group on $2g+2$ letters.

For any integer $n,$ let $\H_g[n]$ denote the $\Z[\zeta_n][1/n]$-scheme which is the coarse moduli space of pairs $(C,{\mathcal B})$ consisting of a hyperelliptic curve $C$ together with a symplectic basis $\mathcal B$  for $J(C)[n]$. As above, for $n \geq 3$ this is a fine moduli space which is smooth over $\Z[\zeta_n][1/n].$
\end{para}

\begin{thm}\label{thm:edhyper} Let $g,n \geq 2$, and let $p$ be any odd prime dividing $n.$  Then
    \[
        \ed(\H_g[n]/\H_g; p)=\dim \H_g=2g-1.
    \]
\end{thm}
\begin{proof} There is a natural map of $\Z[1/2]$-schemes $\H_g \rightarrow \M_g$ which is generically an injective immersion on every fibre over $\Z[1/2],$ as a general hyperelliptic curve has only one non-trivial automorphism \cite{Poonen}, Thm 1.
By the Torelli theorem and \cite{OortSteenbrink} Cor.~3.2, the map
$\M_g \rightarrow \A_g$ is also generically an injective immersion on every fibre over $\Z[1/2],$ and thus so is $\H_g \rightarrow \A_g.$

Now the Jacobian of a hyperelliptic curve over the generic point of $\H_g$ has automorphism group $\{\pm 1\}$ \cite{Matsusaka} p. 790, and $\H_g$ meets the ordinary locus of $\A_g\otimes \F_p$ by \cite{GlassPries}, Theorem 1.
Hence the theorem follows from Corollary \ref{cor:essdimAg}.
\end{proof}

\begin{para}\label{para:p=2}We remark that when $g=2,$ Theorem \ref{thm:edhyper} extends to $p=2,$ as this is a special case
of Theorem \ref{thm:genusg}. However, an extension to $p=2$ is not possible when $g> 2.$
To explain this, recall that for a finite group $G,$ and a prime $p,$ $\ed(G;p)$ denotes the supremum of $\ed(Y/X;p)$ taken over 
all $G$-covers $Y/X$ of finite type $K$-schemes, for any characteristic $0,$ algebraically closed field $K$ 
(the definition being independent of $K$).
The covering 
$$\M_{0,2g+2} \rightarrow \M_{0,2g+2}/S_{2g+2} \iso \H_g$$
 is a component of  $\H_g[2];$ for $g>2$, the cover is disconnected, and all components are isomorphic.\footnote{The monodromy of $\H_g[2]\to\H_g$ was computed by Jordan \cite[p. 364, \S 498]{Jo} to factor as $\SB_{2g+2}\onto S_{2g+2}\into\Sp_{2g}(\F_2)$, where $\SB_{2g+2}=\pi_1(\H_g)$ denotes the spherical braid group. See also \cite[p. 125]{Di}, or for  a more recent treatment, see the $q=2$ case of \cite[Theorem 5.2]{Mc}. The connected components of the cover are in bijection with the cosets $\Sp_{2g}(\F_2)/S_{2g+2}$. The equivalence of the components follows from the monodromy computation.} We conclude that
\begin{align*}
    \ed(\H_g[2]/\H_g;2) &= \ed(\M_{0,2g+2}/\H_g;2) \\
    & = \ed(S_{2g+2};2)\\
    & = g+1< 2g - 1
\end{align*}
where the second equality follows from the versality of $\M_{0,2g+2}$ for $S_{2g+2}$, and the third follows from \cite[Corollary 4.2]{MR}.

The lower bound $\ed(\H_g[2]/\H_g;2) \geq g+1$ can actually be recovered using the techniques of this paper. The point is that although
$\H_g \rightarrow \A_g$ is not generically an immersion in characteristic $2,$ one can show that the image of the map on tangent spaces
at a generic point has dimension $g+1.$ We are grateful to Aaron Landesman for showing us this calculation.
 \end{para}

\section{Essential dimension of congruence covers }\label{s:congcov}

\subsection{Forms of reductive groups}
In this subsection we prove a (presumably well known) lemma showing that for a reductive group
over a number field one can always find a form with given specializations at finitely many places.

\begin{para} Let $F$ be a number field and $G = G^{\ad}$ an adjoint, connected reductive group over $F.$
We fix algebraic closures $\bar F$ and $\bar F_v$ of $F$ and $F_v$ respectively, for every finite place $v$ of $F,$
as well as embeddings $\bar F \hookrightarrow \bar F_v.$

Recall \cite[XXIV, Thm.~1.3]{SGA3} that the
{\em automorphism group scheme} of $G$ is an extension
\begin{equation}\label{eq:autredgp}
1 \rightarrow G \rightarrow \Aut(G) \rightarrow \Out(G) \rightarrow 1
\end{equation}
where $\Out(G)$ is a finite group scheme.
If $G$ is split, then this extension is split and $\Out(G)$ is a constant
group scheme which can be identified with the group of automorphisms of the Dynkin diagram of $G.$

We will also make use of the notion of the {\em fundamental group} $\pi_1(G)$ \cite{Borovoi}.
This is a finite abelian group equipped with a $\Gal(\bar F/F)$-action. As an \'etale sheaf on $\Spec F,$ 
one has $\pi_1(G)\otimes \mu_n \iso \ker(G^{\sc}\rightarrow G)$ where $G^{\sc}$ is the simply connected 
cover of $G^{\der},$ and $n$ is the order $\ker(G^{\sc} \rightarrow G).$
\end{para}

\begin{lemma}\label{lem:twists} Let $G$ be a split, adjoint connected reductive group over $F,$ and $S$
a finite set of places of $F.$ The natural map of pointed sets
$$ H^1(F,\Aut(G)) \rightarrow \prod_{v \in S} H^1(F_v, \Aut(G)) $$
is surjective.
\end{lemma}
\begin{proof} Recall the following facts about the cohomology of reductive groups over global and local fields \cite{Kottwitzelliptic}:
Let $H$ be an adjoint connected reductive group over $F.$
For any place $v$ of $F$, there is a map
$$ H^1(F_v, H) \rightarrow \pi_1(H)_{\Gal(\bar F_v/F_v)},$$
which is an isomorphism if $v$ is finite.
For any finite set of places $T$ of $F,$ consider the composite map
$$\xi: \prod_{v \in T} H^1(F_v, H) \rightarrow \prod_{v\in T} \pi_1(H)_{\Gal(\bar F_v/F_v)} \rightarrow
\pi_1(H)_{\Gal(\bar F/F)}.$$
Then by \cite[\S 2.2]{Kottwitzelliptic} $(x_v)_{v\in T} \in \prod_{v \in T} H^1(F_v, H)$ is in the image of $H^1(F,H)$ if $\xi((x_v)) = 0.$ Applying this to $T = S \cup \{v_0\}$ for some finite place $v_0 \notin S,$ we see that
\begin{equation}\label{eqn:innerforms}
H^1(F,H) \rightarrow \prod_{v \in S} H^1(F_v, H)
\end{equation}
is surjective.

Next we remark that the map
$$H^1(F, \Out(G)) \rightarrow \prod_{v \in S} H^1(F_v, \Out(G))$$
is surjective. Indeed, since $G$ is split a class in $H^1(F,\Out(G))$ is simply a conjugacy class of
maps $\Gal(\bar F/F) \rightarrow \Out(G),$ and similarly for the local classes, so this follows from
\cite[Prop.~3.2]{Cal}.

Now let $(x_v) \in \prod_{v \in S} H^1(F_v, \Aut(G))$ and let  $(\bar x_v) \in \prod_{v \in S} H^1(F_v, \Out(G))$ be the image of $(x_v).$ By what we have seen above, there exists $\bar x \in H^1(F, \Out(G))$ mapping to $(\bar x_v).$
Since we are assuming $G$ is split, (\ref{eq:autredgp}) is a split extension, so there is a $x \in H^1(F,\Aut(G))$
mapping to $\bar x.$ Let $H$ be the twist of $G$ by $x.$ Recall that this means that, if we choose a cocycle $x = (x_\sigma)_{\sigma \in \Gal(\bar F/F)}$ representing $x,$ then there is an isomorphism $\tau: G \iso H$ over $\bar F,$ such for $g \in G(\bar F)$ and $\sigma \in \Gal(\bar F/F),$ we have $\tau(\sigma(g)) = (\sigma(\tau(g)))^{x_{\sigma}}.$
We have a commutative diagram

$$\xymatrix{
H^1(F,\Aut (G)) \ar[r]\ar[d]^\sim_{\tau} & \prod_{v\in S} H^1(F_v,\Aut(G)) \ar[d]^\sim_{\tau|F_v} \\
H^1(F,\Aut(H)) \ar[r] & \prod_{v\in S} H^1(F_v,\Aut(H))
}
$$
such that the vertical maps send $x$ and $(x|_{F_v})_v$ to the trivial classes in the bottom line.
Thus it suffices to show that
$$ H^1(F,H) \rightarrow  \prod_{v\in S} H^1(F_v, H) $$
is surjective, which we saw above.
\end{proof}

\subsection{Shimura varieties}
In this subsection we apply the results of Section \ref{sec:moduli} to compute the essential dimension for congruence covers of Shimura varieties. This will be applied in the next subsection to give examples of congruence covers of locally symmetric varieties where our techniques give a lower bound on the essential dimension. Since our aim is to give lower bounds on essential dimension, it may seem odd that we work with the formalism of Shimura varieties, rather than the locally symmetric varieties which are their geometrically connected components. However, many of the results we need are in the literature only in the former language, and it would take more effort to make the (routine) translation.

\begin{para} Recall \cite[\S 1.2]{DeligneCorvallis} that a {\it Shimura datum} is a pair $(G,X)$ consisting of a connected reductive group $G$ over $\Q,$ and a $G(\RR)$ conjugacy class of maps $h: \CC^\times \rightarrow G(\RR).$ This data is required to satisfy certain properties which imply that the commutant of $h(\CC^\times)$ is a
subgroup $K_{\infty} \subset G(\RR)$ whose image in $G^{\ad}(\RR)$ is maximal compact and  $X = G(\RR)/K_\infty$ is a Hermitian symmetric domain.

Let $\AA$ denote the adeles over $\Q$ and $\AA_f$ the finite adeles. Let $K \subset G(\AA_f)$ be a compact open subgroup. The conditions on $(G,X)$ imply that for $K$ sufficiently small, the quotient
$$ \Sh_K(G,X) = G(\Q) \backslash X \times G(\AA_f)/K $$
has a natural structure of (the complex points of) an algebraic variety over a number field $E = E(G,X) \subset \CC,$ called the reflex field of $(G,X),$ which does not depend on $K.$ We denote this algebraic variety by the same symbol, $\Sh_K(G,X).$

Now let $V_{\Z} = \Z^{2g}$ equipped with a perfect symplectic form $\psi.$ Set $V = V_{\Z}\otimes_{\Z}\Q$ and
$\GSp = \GSp(V,\psi).$ We denote by $S^{\pm}$ the conjugacy class of maps $h: \CC^\times \rightarrow \GSp(\RR)$
satisying the following two properties
\begin{enumerate}
\item The action of the real Lie group $\CC^\times$ on $V_{\CC}$ gives rise to a Hodge structure of type $(-1,0)$ $(0,-1):$
$$ V_{\CC} \iso V^{-1,0} \oplus V^{0,-1}.$$
\item The pairing $(x,y) \mapsto \psi(x,h(i)y)$ on $V_{\RR}$ is positive or negative definite.
\end{enumerate}

Then $(\GSp, S^{\pm})$ is a Shimura datum called the {\it Siegel datum}, and $\Sh_K(\GSp, S^{\pm})$ has an interpretation as the moduli space of principally polarized abelian varieties with suitable level structure.

We say that $(G,X)$ is of {\em Hodge type} if there is a map $\iota: G \hookrightarrow \GSp$ of reductive groups over $\Q$
which induces $X \rightarrow S^{\pm}.$ Given any compact open subgroup $K \subset G(\AA_f)$ there exists
a $K' \subset \GSp(\AA_f)$ such that the above maps induce an embedding of Shimura varieties
\cite[Prop.~1.15]{DeligneTravaux}

$$ \Sh_K(G,X) \hookrightarrow \Sh_{K'}(\GSp, S^{\pm}).$$
\end{para}

\begin{para} Now fix a prime $p,$ and suppose that $G$ is the generic fibre of a reductive group $G_{\Z_{(p)}}$
over $\Z_{(p)}.$ If no confusion is likely to result we will sometimes write simply $G$ for $G_{\Z_{(p)}}.$
We take $K $ to be of the form $K_pK^p$ where $K_p = G(\Z_p)$ and $K^p \subset G(\AA_f^p),$ where $\AA_f^p$ denotes the finite adeles with trivial $p$-component.
Under these conditions, $p$ is unramified in $E,$ and for any prime $\lambda|p$ of $E,$ $\Sh_K(G,X)$ has a canonical smooth model
over $\O_{E_{\lambda}}$ \cite[Thm.~2.3.8]{Kisin}, \cite[Thm.~1]{KimMadapusi}, which we will denote by $\SSh_K(G,X).$
In particular, we may apply this to $\Sh_{K'}(\GSp, S^{\pm})$ if we take $K' = K'_pK^{\prime p}$ with
$K'_p = \GSp(V_{\Z},\psi)(\Z_p).$

Given $G_{\Z_{(p)}}$ and $(G,X)$ of Hodge type, we may always choose $(V,\psi),$ $\iota$ and $K'$ with
$K'_p = \GSp(V_{\Z},\psi)(\Z_p)$ such that $\iota$ induces a map of smooth $\O_{E_{\lambda}}$-schemes
$$ \iota: \SSh_K(G,X) \hookrightarrow \SSh_{K'}(\GSp,S^{\pm})$$
which is locally on the source an embedding \cite[4.1.5]{KisinPappas}, \cite[Prop.~2.3.5]{Kisin}.
That is, if $x \in \SSh_K(G,X)$ is a closed point, and $y = \iota(x),$
then the complete local ring at $x$ is a quotient of the  complete local ring at $y.$
In particular, there is an open subscheme of $\SSh_K(G,X),$ whose special fibre is dense in
$\SSh_K(G,X)\otimes \F_p,$ such that the restriction of $\iota$ to this open subscheme is a locally closed embedding.
Fix such choices. As in \S \ref{sec:moduli}, we denote by $\A$ the universal abelian scheme over $\SSh_{K'}(\GSp,S^{\pm})$.
Then we have
\end{para}

\begin{lemma}\label{lem:edineq} Suppose that the reflex field $E,$ admits a prime $\lambda|p$ with residue field $\F_p.$
Then a Galois closure of
the \'etale local system $\A[p]|_{\Sh_K(G,X)}\to \Sh_K(G,X)$ is given by a congruence cover $\SSh_K(G,X)_p\to \SSh_K(G,X)$,
with monodromy group isomorphic to $G^{\der}(\F_p)$ over every geometrically connected component of $\Sh_K(G,X).$

Moreover,
$$ \ed(\Sh_K(G,X)_p\to \Sh_K(G,X); p) = \dim_{\CC} X.$$
\end{lemma}
\begin{proof} Let $S$ and $S'$ denote geometrically connected components of $\Sh_K(G,X)$ and $\Sh_{K'}(\GSp,S^{\pm})$
respectively with $S \subset S'.$
The \'etale local system $\A[p]|_{S_{\bar \Q}}$ corresponds to a representation
$$ \rho_G: \pi_1(S_{\bar \Q}) \rightarrow \pi_1(S'_{\bar \Q}) \rightarrow \GSp(\F_p)$$
where $\GSp = \GSp (V_{\Z},\psi),$ as above, and $\rho_G$ has image $G^{\der}(\F_p)$:
By the smooth base change theorem and the comparison between the algebraic \'etale and topological
fundamental groups for varieties over $\CC,$ it suffices to check that the composite of the maps of topological fundamental groups
$$ \pi_1^{\top}(S(\CC)) \rightarrow \pi_1^{\top}(S'(\CC)) \rightarrow \GSp(\F_p)$$
has image $G^{\der}(\F_p).$
This sequence of maps may be identified with \cite[2.1.2, 2.0.13]{DeligneCorvallis}
$$ \Gamma \rightarrow \Gamma' \rightarrow \GSp(\F_p) $$
where $\Gamma$ and $\Gamma'$ are $T$-congruence subgroups of $G^{\der}(\Q)$ and
$\Sp(\Q)$-respectively, for some finite set of finite places $T$ not containing $p.$
In particular, the image of the composite map is $G^{\der}(\F_p).$

By  Lemma \ref{lem:equiedlem} we have
$$ \ed(\Sh_K(G,X)_p\to \Sh_K(G,X); p) = \ed(\A[p]|_{S_{\bar \Q}}/S_{\bar \Q}; p).$$
We now consider the map of integral models $\iota,$ corresponding to the prime $\lambda|p$ of the lemma.
Since  $\lambda$ has residue field $\mathbb F_p,$ every component of the image $\iota$ meets the ordinary locus of
$\SSh_{K'}(\GSp,S^{\pm})$ by \cite[Thm.~1.1]{Wort}.  Now using that for some $N$ with $(N,p) = 1,$ there is a surjective map $\A_{g,N} \rightarrow \Sh_{K'}(\GSp,S^{\pm}),$
and Theorem \ref{thm:essdimAg}, we conclude
$$ \ed(\A[p]|_{S_{\bar \Q}}/S_{\bar \Q}; p) = \dim S = \dim X.$$
\end{proof}

\subsection{Congruence covers} It will be more convenient to state the results of this subsection in terms of locally symmetric varieties. These are geometrically connected components of the Shimura varieties discussed in the previous subsection.

\begin{para} Let $G$ be a semisimple, almost simple group over $\Q.$ We will assume that $G$ is of classical type, so that
(the connected components of) its Dynkin diagram are of type $A,B,C$ or $D.$

Let $K_\infty \subset G(\RR)$ be a maximal compact subgroup. We will assume that $X = G(\RR)/K_\infty$ is a Hermitian
symmetric domain. The group $G^{\ad}(\RR)$ is a product of simple groups $G_i(\RR),$ for $i$ in some index set $I.$
We denote by $I_{nc}$ (resp.~$I_c$) the set of $i$ with $G_i$ non-compact (resp.~compact).
Then $X$ is a product of the irreducible Hermitian symmetric domains $X_i = G_i(\RR)/K_i,$  for $i \in I_{nc},$ where
$K_i \subset G_i(\RR)$ is maximal compact. We use Deligne's notation \cite{DeligneCorvallis}
for the classification of these irreducible Hermitian symmetric domains. Since we are assuming $G$ is of classical type, for  $i \in I_{nc},$
$X_i$ is  of type $A, B, C,$ $D^{\RR}$ or $D^{\mathbb H}.$
The group $G_i(\RR)$ is either the adjoint group of $U(p,q)$ in case of type $A,$ of $\Sp(2n)$ in the case of type $C,$ of $\SO(n,2)$ in case of type $B$ or $D^{\RR},$ and of $\SO^*(2n),$ an inner form of $\SO(2n),$ if $G_i$ is of type $D^{\mathbb H}.$

Since $G$ is almost simple, for $i \in I_{nc}$ the $X_i$ are all of the same type, except possibly if $G$ is of type $D,$ in which
case it is possible that both factors of type $D^{\RR}$ and $D^{\mathbb H}$ occur among the $X_i.$
We will say that $G$ is of {\em Hodge type} if all the factors $X_i$ are of the same type $A,B,C,D^{\RR}, D^{\mathbb H}$
and the following condition holds: $G$ is simply connected, unless the $X_i$ are of type $D^{\mathbb H},$ in which case
$G(\CC)$ is a product of special orthogonal groups.
\end{para}

\begin{para}
The Dynkin diagram $\Delta(G)$ is equipped with a set of vertices $\Sigma(G)$ which is described as follows
(cf.~\cite{DeligneCorvallis} \S 1.2, 1.3). For $i \in I_{nc},$ $K_i \subset G_i$ is the centralizer of a rank $1$ compact torus
$U(1) \subset G_i,$ which is the center of $K_i.$ Thus there are two cocharacters $h,h^{-1}: U(1) \rightarrow G_i$ which identify
$U(1)$ with this compact torus. These cocharacters are miniscule, and each corresponds to a vertex of $\Delta(G_i).$ The two vertices
are distinct exactly when $h,h^{-1}$ are not conjugate cocharacters. In this case, they are exchanged by the {\em opposition involution} of
$\Delta(G_i),$ which also gives the action of complex conjugation on $\Delta(G_i).$
We set $\Sigma(G)$ to be the union of all the vertices above. Thus $\Sigma(G) \cap \Delta(G_i)$ is empty if $i \in I_c,$
and consists of one or two vertices if $i \in I_{nc}.$ In the latter case it consists of two vertices if and only if $G_i(\RR)$ is either
the adjoint group of $U(p,q)$ with $p \neq q,$ or of $\SO^*(2n)$ with $n$ odd.
\end{para}

\begin{para} Fix an embedding $\bar \Q \hookrightarrow \Bbb C.$ The Galois group $\Gal(\bar \Q/\Q)$ acts on $\Delta(G).$
We consider a subset $\Sigma \subset \Sigma(G)$ such that $\Delta(G_i) \cap \Sigma$ consists of one element for $i \in I_{nc}.$.

We call $G$ {\em $p$-admissible} if $G$ splits over an unramified extension of $\Q_p,$ and for some embedding
$\bar \Q \hookrightarrow \bar \Q_p,$ and some choice of $\Sigma,$ the action of $\Gal(\bar \Q_p/\Q_p)$ leaves $\Sigma$ invariant.
This definition may look slightly odd; it will be used to guarantee that the reflex field of a Shimura variety built out of $G$ has at least one prime where the Shimura variety has a non-empty ordinary locus.
\end{para}

\begin{para}
For any reductive group $H$ over $\Q$ a {\em congruence subgroup} $\Gamma \subset H(\Q)$ is a group of the form $H(\Q) \cap K$ for some compact open subgroup $K \subset H(\AA_f).$ An {\em arithmetic lattice} $\Gamma \subset H(\Q)$ is a finite index subgroup of a
congruence subgroup. If $i:G' \rightarrow G$ is a map of reductive groups whose kernel is in the center of $G',$ and $\Gamma' \subset G'(\Q)$ is an arithmetic lattice then $i(\Gamma') \subset G(\Q)$ is an arithmetic lattice.

If $\Gamma \subset G(\Q)$ is an arithmetic lattice which acts freely on $X$ then $M_{\Gamma}: = \Gamma\backslash X$ has a natural structure of algebraic variety over $\bar \Q$ \cite[\S 2]{DeligneCorvallis}. For any arithmetic lattice $\Gamma$ there is a finite index subgroup which acts freely on $X.$ For a Shimura datum $(G,X)$ the geometrically connected components of $\Sh_K(G,X)$ have the form
$\Gamma\backslash X,$ where $\Gamma \subset G^{\ad}(\Q)$ is the image of a congruence subgroup of $G^{\der}(\Q);$ this was already used in the proof of Lemma \ref{lem:edineq}.

Now suppose that $G$ is almost simple and $G$
splits over an unramified extension of $\Q_p.$ Then $G$ extends to a reductive group $G_{\Z_p}$ over $\Z_p.$
Let $K = K^pK_p \subset G(\AA_f)$ and $K_1 = K^pK^1_p \subset G(\AA_f)$ be compact open, with $K_p\subset G(\AA_f^p),$
$K_p = G_{\Z_p}(\Z_p)$ and $K^1_p = \ker(G_{\Z_p}(\Z_p) \rightarrow G_{\Z_p}(\F_p)).$
Let $\Gamma = G(\Q) \cap K$ and $\Gamma_1 = G(\Q) \cap K_1.$ We call a covering of the form
$\Gamma_1\backslash X \rightarrow \Gamma \backslash X$ a {\em principal $p$-level covering.}
\end{para}

\begin{thm}\label{thm:essdimcong}
Let $G'$ be an almost simple group of Hodge type which is $p$-admissible, and let $X = G'(\RR)/K_{\infty}.$
Then for any principal $p$-level covering.
$\Gamma_1\backslash X \rightarrow \Gamma \backslash X$ we have
$$ \ed(\Gamma_1\backslash X \rightarrow \Gamma \backslash X; p) = \dim X. $$
\end{thm}
\begin{proof}

We slightly abuse notation and write $G^{\ad}$ for $G^{\prime \ad}.$
Let $\Sigma \subset \Sigma(G')$ be a subset of the form described above. This corresponds to a $G^{\ad}(\RR)$-conjugacy class of cocharacters $h: U(1) \rightarrow G^{\ad},$ which we denote by $X^{\ad}.$
Then $(G^{\ad}, X^{\ad})$ is a Shimura datum and its reflex field corresponds to the subgroup of $\Gal(\bar \Q/\Q)$ which takes $\Sigma$ to itself  \cite[Prop.~2.3.6]{DeligneCorvallis}. Since $G'$ is $p$-admissible, there is a choice of $\Sigma,$ and a prime $\lambda'|p$ of $E(G^{\ad},X^{\ad})$
with $\kappa(\lambda') = \F_p.$

Then one sees using \cite[Prop.~2.3.10]{DeligneCorvallis} that one can choose a Shimura datum of Hodge type $(G,X)$ with $G^{\der} = G'$
and adjoint Shimura datum $(G^{\ad}, X^{\ad}),$ and so that
all primes of $E(G^{\ad},X^{\ad})$ above $p$ split completely in $E(G,X).$ In particular, any prime $\lambda|\lambda'$ of $E(G,X)$
has residue field $\F_p.$ We have verified that $(G,X)$ satisfies the hypotheses of Lemma \ref{lem:edineq}. The theorem follows by restricting the map of that lemma to geometrically connected components.
\end{proof}

\begin{para} We can make the condition of $p$-admissibility of $G'$ in Theorem \ref{thm:essdimcong} somewhat more explicit, if we assume that
$G^{\prime \ad}(\RR)$ has no compact factors.
\end{para}

\begin{cor}\label{cor:essdimcong} Let $G'$ be an almost simple group which splits over an unramified extension of $\Q_p.$
Suppose either that
\begin{enumerate}
\item $G'$ splits over $\Q_p,$ or
\item The irreducible factors of $G^{\prime \ad}(\RR)$ are all isomorphic to the adjoint group of one of $U(n,n),$ $\SO(n,2)$ with $n \neq 6,$ or $\Sp(2n)$ for some positive integer $n.$
\end{enumerate}
Then $G'$ is $p$-admissible, and for any principal $p$-level covering
$\Gamma_1\backslash X \rightarrow \Gamma \backslash X$ we have
$$ \ed(\Gamma_1\backslash X \rightarrow \Gamma \backslash X; p) = \dim X.$$
\end{cor}
\begin{proof} If $G'$ splits over $\Q_p$ then $\Gal(\bar \Q_p/\Q_p)$ acts trivially on $\Delta(G),$ and so leaves any choice of $\Sigma$ stable.
For (2), one checks using the classification of \cite{DeligneCorvallis} that in each of these cases, $\Sigma = \Sigma(G),$
and a vertex $v \in \Sigma(G)$ is stable by any automorphism of the connected component of $\Delta(G)$ containing $v.$
It follows that  $\Gal(\bar \Q/\Q)$ leaves $\Sigma(G)$ stable.
\end{proof}

\begin{para} The above results give examples of coverings for which one can compute the essential $p$-dimension. It may be of interest to compare these numbers with the essential $p$-dimension of the corresponding group, which can in principle be computed using the Karpenko-Merkjurev theorem \cite{KarMer}.

We call a reductive group $H$ almost absolutely simple if $H$ is semisimple and $H^{\ad}$ is absolutely simple. (That is, it remains simple
over an algebraic closure).We have the following result.
\end{para}

\begin{prop}\label{prop:oldmainthm} Let $H$ be a classical, almost absolutely simple group over $\F_q,$ with $q = p^r.$
Then there exists a Hermitian symmetric domain $X$ attached to an adjoint $\Q$-group $G,$  and arithmetic lattices
$\Gamma' \subset \Gamma \subset G(\Q)$ corresponding to a principal $p$-covering, with $\Gamma/\Gamma' \iso H(\F_q),$ such that
$$ e = \ed(\Gamma'\backslash X \rightarrow \Gamma\backslash X; p)$$
satisfies

\begin{itemize}
\item If $H$ is a form of $\SL_n$ which is split if $n$ is odd, then $e = r\lfloor \frac{n^2}4 \rfloor.$
\item If $H$ is $\Sp_{2n}$ then $e = r(\frac{n^2+n}2).$
\item If $H$ is a split form of  $\SO_{2n}$ then $e = r(\frac{n^2-n}2).$
\item If $H$ is a form of $\Spin_n$ and $H$ is not of type $D_4,$ then $e(H(\F_q)) = r(n-2).$
\end{itemize}
\end{prop}
\begin{proof} Let $\bar G = \Res_{\F_q/\F_p} H.$ There is a unique (up to canonical isomorphism)
connected reductive group $G_{\Z_p}$ over $\Z_p$ with $G_{\Z_p}\otimes\F_p = \bar G.$

We (may) now assume that $H$ is one of the four types listed, and in each of these cases we define a
semisimple Lie group $G_{\RR}$ over $\RR$ as follows. If $H$ is a form of $\SL_n,$
we take $G_{\RR}$ to be $\SU(\frac n 2, \frac n 2)^r$ if $n$ is even and $\SU(\frac {n-1}2, \frac {n+1} 2)^r$
if $n$ is odd. If $H$ is $\Sp_{2n}$ we take $G_{\RR} = \Sp_{2n}^r.$ If $H$ is a form of $\SO_{2n}$
we take $G_{\RR}$ to be $\SO^*(2n)^r,$ the inner form of (the compact group) $\SO(2n)$ which gives rise to
the Hermitian symmetric domain of type $D^{\mathbb H}_n$ (cf.~\cite[1.3.9,1.3.10]{DeligneCorvallis}). If $H$ is a form of $\Spin_n$ we take $G_{\RR}$ to be $\Spin(n-2,2)^r.$

In all cases $G_{\RR}$ and $G_{\Z_p}$ are forms of the same split group. Thus by Lemma \ref{lem:twists}
there exists a semisimple reductive group $G$ over $\Q,$ which gives rise to $G_{\RR}$ and $G_{\Z_p}$
over $\RR$ and $\Z_p$ respectively. By assumption $G'$ is of Hodge type, and we now check that it can be chosen to be $p$-admissible.
This is necessarily the case by Corollary \ref{cor:essdimcong}, except when $n$ is odd and $H$ is a form of $\SL_n$ or $H$ is a form of $\SO(2n).$
In these cases, we are assuming that $H$ is a split form, so $\Gal(\bar \Q_p/\Q_p)$ permutes the components of $\Delta(G)$ simply transitively.
If $H$ is a form of $\SL_n$ or $\SO(2n)$ with $n$ odd, then $\Delta(G_i)$ contains two points in $\Sigma(G)$
(the opposition involution is nontrivial on $\Delta(G_i)$ in these cases),
and we can take $\Sigma \subset \Sigma(G)$ to be a $\Gal(\bar \Q_p/\Q_p)$-orbit of any point $v \in \Sigma(G).$

When $H$ is a form of $\SO(2n)$ with $n$ even, then the opposition involution is trivial on $\Delta(G_i),$ and hence so is the action of complex conjugation.  The set $\Sigma(G)$ meets each component of $\Delta(G)$ in one vertex. Let $\Sigma_1$ be the $\Gal(\bar \Q_p/\Q_p)$-orbit
of any vertex in $\Sigma(G).$ There is an inner form $G_{1,\RR}$ of $G$ over $\RR$ such that $\Sigma(G_{1,\RR}) = \Sigma_1.$
(Note that the Dynkin diagrams of inner forms are identified so this makes sense.) Explicitly, let $a \in \Out(G)$ be an automorphism
which preserves the connected components of $\Delta(G)$ and such that $a(\Sigma(G)) = \Sigma_1.$
(Such an $a$ is unique except if $H$ is of type $D_4.$)
Then $G_{1,\RR}$ is given by twisting $G$ by the cocycle $\tilde a\sigma(\tilde a)^{-1}$ where
$\tilde a \in \Aut(G)(\CC)$ lifts $a.$ Using the surjection (\ref{eqn:innerforms}), we see that there is an inner twisting $G_1$ of $G$ over
$\Q$ which is isomorphic to $G_{1,\RR}$ over $\RR$ and to $G$ over $\Q_p,$ as an inner twist. As $\Sigma(G_1) = \Sigma_1$
is stable by $\Gal(\bar \Q_p/\Q_p),$ $G_1$ is $p$-adimissible.
Thus, the proposition follows from Theorem \ref{thm:essdimcong} and the formulae for the dimensions of Hermitian symmetric domains.
\end{proof}

\begin{cor}\label{cor:Lieexamples} Let $H$ be a classical, absolutely simple group over $\F_q,$ with $q = p^r.$ Then there is a congruence $H(\F_q)$-cover of locally symmetric varieties $Y' \rightarrow Y$ such that $e:= \ed_K(Y'/ Y; p)$ satisfies :
\begin{itemize}
\item If $H$ is a form of $\PGL_n$ which is split if $n$ is odd, then $e = r\lfloor \frac{n^2}4 \rfloor.$
\item If $H$ is $\PSp_{2n}$ then $e = r(\frac{n^2+n}2).$
\item If $H$ is a split form of  $\PO_{2n}$ then $e = r(\frac{n^2-n}2).$
\item If $H$ is a form of $\PO_n$ and $H$ is not of type $D_4,$ then $e = r(n-2).$
\end{itemize}
\end{cor}
\begin{proof} This follows immediately from Proposition \ref{prop:oldmainthm} and Lemma \ref{lem:brauer}.
\end{proof}


\end{document}